\documentclass[11pt,reqno]{amsart}
\usepackage[foot]{amsaddr}

\usepackage{hyperref, cleveref}
\usepackage{enumitem}
\usepackage{amsfonts, amssymb}
\usepackage{tikz}
\usepackage{diagbox}
\usetikzlibrary{decorations.markings}
\usepackage{mathtools}
 
 \tikzset{->-/.style={decoration={
  markings,
  mark=at position .5 with {\arrow{>}}},postaction={decorate}}}
\numberwithin{figure}{section}
\numberwithin{equation}{section}
\numberwithin{table}{section}

\newtheorem{thm}{Theorem}[section] 
\newtheorem{lem}[thm]{Lemma}
\newtheorem{cor}[thm]{Corollary}
\newtheorem{prop}[thm]{Proposition}
\newtheorem{alg}[thm]{Algorithm}
\newtheorem{conj}[thm]{Conjecture}

\theoremstyle{definition}
\newtheorem{defn}[thm]{Definition}
\newtheorem{exm}[thm]{Example}

\newcommand{\zz}{\mathbb{Z}}

\newcommand{\qq}{\mathbb{Q}}

\newcommand{\slq}{{\rm SL_2}(\qq)}
\newcommand{\slqp}{{\rm SL_2}(\qq_p)}

\newcommand{\ist}{{\rm Isom}(T)}

\title{Discrete and free groups acting on locally finite trees}
\author{Matthew J. Conder}
\address{Department of Mathematics, University of Auckland, New Zealand}
\email{matthew.conder@auckland.ac.nz}

\begin{document}

\begin{abstract}
We present an algorithm to decide whether or not a finitely generated subgroup of the isometry group of a locally finite simplicial tree is both discrete and free. The correctness of this algorithm relies on the following conjecture: every `minimal' $n$-tuple of isometries of a simplicial tree either contains an elliptic element or satisfies the hypotheses of the Ping Pong Lemma. We prove this conjecture for $n=2,3$, and show that it implies a generalisation of Ihara's Theorem.
\end{abstract}

\maketitle

\vspace{-0.5cm}

\section{Introduction}

Let $X=(g_1, \dots, g_n)$ be an $n$-tuple of isometries of a simplicial tree $T$. A theorem attributed to Weidmann (see \cite[Theorem 4]{AG}) asserts that the group generated by $X$ is either free, or contains an elliptic isometry of $T$. The proof, however, gives no practical method of distinguishing between these two cases.

In this paper, we conjecture that every $n$-tuple $X$ which is `minimal' (in the sense that certain product replacements do not decrease a particular sum of translation lengths) either contains an elliptic isometry or satisfies the hypotheses of the Ping Pong Lemma. We prove this conjecture for $n=2,3$, and report significant computational evidence that it holds for $n>3$. It implies the existence of a practical method to decide between the two outcomes of the theorem of Weidmann, and to exhibit such an elliptic isometry in the second case.

If $T$ is locally finite and its isometry group $\ist$ is equipped with the topology of pointwise convergence, then this conjecture additionally gives an algorithm to decide whether or not a finitely generated subgroup of $\ist$ is both discrete and free. This significantly generalises the algorithm given in \cite{C} for 2-generated subgroups of $\ist$. In fact, our conjecture implies the following generalisation of Ihara's Theorem \cite[Theorem 1]{Ihara}: a finitely generated subgroup of $\ist$ is both discrete and free if and only if it contains no elliptic element. It immediately follows that a discrete subgroup of $\ist$ is locally free if and only if it contains no elliptic element.

\section{Main Results}

Let $X=(g_1, \dots, g_n)$ be an $n$-tuple of isometries of a simplicial tree $T$. A \textit{product replacement} is where we replace some $g_i$ by either $g_jg_i$ or $g_ig_j^{-1}$ (for $i \neq j$). This is an example of a Nielsen transformation, and it preserves the subgroup of $\ist$ generated by $X$. Given $j \in \{1, \dots, n\}$ and subsets $S_1, S_2$ of $ \{1, \dots, n \}\backslash  \{ j \}$, we will denote by $X^{j}_{S_1,S_2}$ the $n$-tuple obtained from $X$ by performing the following product replacements:
\begin{align*}
g_i &\mapsto g_jg_i \;\;\;\;\;\;\;\;\; (\textup{if } i \in S_1); \\
g_i &\mapsto g_ig_j^{-1} \;\;\;\;\;\;\; (\textup{if } i \in S_2).
\end{align*}
Note that if $i \in S_1 \cap S_2$, then the $i$-th element of $X^{j}_{S_1,S_2}$ is $g_jg_ig_j^{-1}$.

\begin{defn}\label{sum}
Let $X=(g_1, \dots, g_n)$ be an $n$-tuple of isometries of a simplicial tree. Define the following sum of translation lengths:

$$L(X)=\sum\limits_{1 \le i \le n} l(g_i) + \sum\limits_{1 \le i < j \le n} l(g_ig_j)+l(g_ig_j^{-1}).$$
\end{defn}

\begin{defn}\label{min}
An $n$-tuple $X=(g_1, \dots, g_n)$ of isometries of a simplicial tree is \textit{minimal} if $L(X)\le L(X^j_{S_1,S_2})$ for every $j \in \{1, \dots, n\}$ and every $S_1, S_2 \subseteq \{1, \dots, n \}\backslash  \{ j \}$.
\end{defn}

\begin{conj}\label{conj}
Let $X=(g_1, \dots, g_n)$ be a $n$-tuple of isometries of a simplicial tree $T$. If $X$ is minimal, then either some $g_i$ is elliptic, or $X$ satisfies the hypotheses of the Ping Pong Lemma.
\end{conj}

Many different statements of the Ping Pong Lemma appear in the literature, but the version we use here is stated in \Cref{background}.

\begin{thm}\label{thm:conj}
Conjecture $\ref{conj}$ holds for $n=2,3$.
\end{thm}

The proof of \Cref{conj} for $n=3$ relies on a case-by-case analysis of the possible interactions between the axes of three hyperbolic elements. The main obstacle to obtaining a full proof of \Cref{conj} is the vast increase in the number of possible interactions between the axes as $n$ increases. However, as discussed in \Cref{sec:conj}, there is significant computational evidence that \Cref{conj} holds.

When equipped with the topology of pointwise convergence, the isometry group $\ist$ has the structure of a topological group \cite[Section 5.B]{CH}. A subgroup of $\ist$ is \textit{discrete} if the corresponding topology is the discrete topology.

\begin{thm}\label{thm:alg}
Let $T$ be a locally finite simplicial tree. If Conjecture $\ref{conj}$ holds, then there is an algorithm to decide whether or not a finitely generated subgroup of $\ist$ is both discrete and free.
\end{thm}

We present such an algorithm in \Cref{sec:alg} and discuss its implementation in {\sc Magma} \cite{magma} in \Cref{algimp}. As a consequence of this algorithm, we obtain the following generalisation of Ihara's Theorem \cite[Theorem 1]{Ihara}:

\begin{thm}\label{ihara}
Let $T$ be a locally finite simplicial tree. If Conjecture $\ref{conj}$ holds, then a finitely generated subgroup of $\ist$ is both discrete and free if and only if it contains no elliptic element.
\end{thm}

A group is \textit{locally free} if every finitely generated subgroup is free. Thus we immediately deduce the following from \Cref{ihara}:

\begin{cor}
Let $T$ be a locally finite simplicial tree. If Conjecture $\ref{conj}$ holds, then a discrete subgroup of $\ist$ is locally free if and only if it contains no elliptic element.
\end{cor}

\section{Background}\label{background}

Let $T$ be a simplicial tree. We will assume (by subdividing edges, if necessary) that every isometry $g$ of $T$ acts without inversions and hence can be classified based on its {\it translation length}. This is the integer
$$l(g)=\min_{x\in V(T)}d(x,gx),$$
where $V(T)$ denotes the vertex set of $T$, and $d$ is the standard path metric on $T$. Note that $l(g)=l(g^{-1})$ and $l(hgh^{-1})=l(g)$ for all such isometries $g, h$ of $T$. If $l(g)=0$, then $g$ is \textit{elliptic} and it fixes a vertex of $T$. If $l(g)>0$, then $g$ is \textit{hyperbolic} and there is an infinite path in $T$ (called the \textit{axis} of $g$) upon which $g$ acts by translations of length $l(g)$. For further background information, see \cite{Serre}.

Suppose $g_i, g_j$ are hyperbolic isometries of $T$ with respective axes $\gamma_i, \gamma_j$. As in \cite{AG}, we define the \textit{projection of $\gamma_j$ onto $\gamma_i$} to be
$${\rm Proj}_{\gamma_i}(\gamma_j)=\{x \in \gamma_i : d(x,\gamma_j)=d(\gamma_i,\gamma_j)\}.$$

Note that ${\rm Proj}_{\gamma_i}(\gamma_j)$ is either the unique vertex of $\gamma_i$ that is closest to $\gamma_j$ (if $\gamma_i \cap \gamma_j = \varnothing$), or the path $\gamma_i \cap \gamma_j$ (if $\gamma_i\cap \gamma_j \neq \varnothing$).

We present a version of the Ping Pong Lemma specifically for isometries of a simplicial tree. It is essentially a reformulation of \cite[Proposition 1.6]{L} using the notation of \cite{AG}.

\begin{lem}[The Ping Pong Lemma]\label{PPL}
Let $X=(g_1, \dots, g_n)$ be an $n$-tuple of hyperbolic isometries of a simplicial tree with axes $\gamma_1, \dots, \gamma_n$. Suppose that 
for each $1 \le i \le n$ there is an open segment $P_i \subseteq \gamma_i$ of length $l(g_i)$ such that
$$\bigcup_{i \neq j}{\rm Proj}_{\gamma_i}(\gamma_j) \subseteq P_i.$$
Then the group $G$ generated by $X$ is free of rank $n$. If $T$ is locally finite, then $G$ is also discrete with respect to the topology of pointwise convergence on $\ist$.
\end{lem}
\begin{proof}
This follows from \cite[Lemma 2.1]{CS}.
\end{proof}

\Cref{PPL} implies that two hyperbolic isometries of $T$ generate a free group if their axes either do not intersect or intersect along a common subpath of length strictly less than both their translation lengths. As noted in \cite{C,P}, the interaction between the axes of two hyperbolic isometries of a simplicial tree can be deduced from the translation length of their product:

\begin{prop}\label{overlap}
Let $g_1, g_2$ be hyperbolic isometries of a simplicial tree, with respective axes $\gamma_1, \gamma_2$. Precisely one of the following holds:
\begin{enumerate}[label={$(\arabic*)$}]
\item $\gamma_1$ and $\gamma_2$ do not intersect and
$$l(g_1g_2)=l(g_1)+l(g_2)+2d(\gamma_1,\gamma_2).$$
\item $\gamma_1$ and $\gamma_2$ intersect with the same orientation and
$$l(g_1g_2)=l(g_1)+l(g_2).$$
\item $\gamma_1$ and $\gamma_2$ intersect with opposite orientations along a (possibly infinite) path of length $\Delta=\Delta(\gamma_1,\gamma_2)\ge 0$ and one of the following holds:
\begin{enumerate}[label={$(\roman*)$}]
\item $\Delta<\min\{l(g_1),l(g_2)\}$ and $l(g_1g_2)=l(g_1)+l(g_2)-2\Delta$;
\item $\Delta>\min\{l(g_1),l(g_2)\}$ and $l(g_1g_2)=|l(g_1)-l(g_2)|$;
\item $\Delta=\min\{l(g_1),l(g_2)\}$, either $\gamma_2$ and $g_1\cdot \gamma_2$ (if $l(g_1) \le l(g_2)$) or $\gamma_1$ and $g_2\cdot \gamma_1$ (if $l(g_1) > l(g_2)$) intersect along a (possibly infinite) path of length $\Delta' \ge 0$, and \\
$$\hspace{2cm} l(g_1g_2)=\left\{
\begin{array}{ll}
|l(g_1)-l(g_2)|-2\Delta' & \textup{ if } \Delta' < \frac{|l(g_1)-l(g_2)|}{2} \\
0 & \textup{ otherwise. }\\
\end{array} 
\right.$$
\end{enumerate}
\end{enumerate}
\end{prop}
\begin{proof}
See \cite[Proposition 3.5]{C}. 
\end{proof}

In \Cref{app}, we include figures which illustrate each case of \Cref{overlap}. We repeatedly use \Cref{overlap} and these figures in the proof of \Cref{thm:conj} in \Cref{proofs}. We will also require the following observation:

\begin{cor}\label{ellfix}
Let $g_1, g_2$ be hyperbolic isometries of a simplicial tree, with respective axes $\gamma_1, \gamma_2$. If $g_1g_2$ is elliptic, then $\gamma_1$ and $\gamma_2$ intersect with opposite orientations along a path $[p,q]$ of length $\Delta \ge \min\{l(g_1),l(g_2)\}$ (where $g_1$ and $g_2^{-1}$ translate $p$ towards $q$) and one of the following holds:
\begin{itemize}
\item $l(g_1)=l(g_2)$, and $g_1g_2$ fixes $q$;
\item $\Delta=\min\{l(g_1),l(g_2)\}$, and $g_1g_2$ fixes a vertex which is at distance $\frac{1}{2}|l(g_1)-l(g_2)|$ from $q$;
\end{itemize}
\end{cor}
\begin{proof}
By \Cref{overlap}, $\gamma_1$ and $\gamma_2$ must intersect with opposite orientations along a path of length $\Delta \ge \min\{l(g_1),l(g_2)\}$. If $l(g_1)=l(g_2)$, then $g_1g_2$ fixes $q$, so suppose that $l(g_1) \neq l(g_2)$. Hence we must be in case $(3)(iii)$ of \Cref{overlap} and $\Delta=\min\{l(g_1),l(g_2)\}$. 

After interchanging the roles of $g_1$ and $g_2$, if necessary, we may suppose that $\Delta=l(g_1)<l(g_2)$. It follows that $g_1g_2$ is elliptic if the path $\gamma_2 \cap g_1 \cdot \gamma_2$ has length $\Delta' \ge \frac{l(g_2)-l(g_1)}{2}$. As in the left-hand diagram of \Cref{bigint2}, if $\Delta'<l(g_2)-l(g_1)$, then $g_1g_2$ fixes the midpoint of the path $[r,g_1g_2r]$, which lies on $\gamma_2$ at distance $\frac{1}{2}(d(q,g_1g_2r)+\Delta')=\frac{1}{2}(l(g_2)-l(g_1))$ from $q$. On the other hand, if $\Delta'\ge l(g_2)-l(g_1)$, then $g_1g_2$ fixes the midpoint of $[q,g_1g_2q]$, which lies on $\gamma_2$ at distance $\frac{1}{2}d(q,g_1g_2q)=\frac{1}{2}(l(g_2)-l(g_1))$ from $q$; see the right-hand diagram of \Cref{bigint2}.
\end{proof}

We conclude this section with an important class of isometries:

\begin{exm}\label{BT}
Let $p$ be a prime and define the \textit{$p$-adic valuation} of $x\in\qq$ to be $v_p(x)=r$, where $x=p^r\frac{a}{b}$ with $p \nmid a,b$. The completion of $\qq$ with respect to the absolute value $|x|_p=p^{-v_p(x)}$ is $\qq_p$, the field of \textit{$p$-adic numbers}. Associated to $\qq_p$ is a $(p+1)$-regular simplicial tree $T_p$, known as the \textit{Bruhat-Tits tree}; see \cite[Chapter II.1]{Serre} for further detail.

The topological group ${\rm Isom}(T_p)$ contains ${\rm PSL_2}(\qq_p)$, where the latter is equipped with the quotient topology inherited from ${\rm SL_2}(\qq_p)$. We therefore may identify elements of ${\rm SL_2}(\qq_p)$ with isometries of $T_p$, and such isometries act without inversions. The translation length of $A \in \slqp$ is
\begin{align}\label{TL}
l(A)=-2\min\{0,v_p({\rm tr}(A))\},
\end{align}
where ${\rm tr(A)}$ denotes the trace of $A$ \cite[Proposition II.3.15]{MS}.
\end{exm}

\section{The algorithm}\label{sec:alg}

In this section we prove Theorems \ref{thm:alg} and \ref{ihara} by presenting an algorithm which, subject to \Cref{conj}, decides whether or not a finitely generated subgroup of the isometry group of a locally finite simplicial tree $T$ is both discrete and free. The algorithm requires a method of computing translation lengths on $T$. In particular, using \Cref{TL}, this algorithm can be applied to finitely generated subgroups of $\slqp$. We first note the following:


\begin{prop}\label{elltree}
Let $T$ be a locally finite simplicial tree. If a subgroup of ${\rm Isom}(T)$ is both discrete and free, then it contains no elliptic element.
\end{prop}
\begin{proof}
See \cite[Proposition 5.1]{C}.
\end{proof}

\begin{alg}\label{alg}
Let $T$ be a locally finite simplicial tree. Given a finitely generated subgroup $H=\langle h_1, \dots, h_n\rangle \le \ist$, we proceed as follows.

If $H$ is both discrete and free of rank $n$, then the algorithm returns {\tt true} and outputs an $n$-tuple of isometries which generates $H$ and satisfies the hypotheses of the Ping Pong Lemma. Otherwise, the algorithm returns {\tt false} and outputs an elliptic element of $H$.
\begin{enumerate}[label={$(\arabic*)$}]
\item Set $g_i=h_i$ for each $i \in \{1, \dots, n\}$ and initialise $X=(g_1,\dots, g_n)$.
\item If $l(g_i)=0$ for some $i \in \{1, \dots, n\}$, then return {\tt false} and the element $g_i$.
\item If $L(X^j_{S_1,S_2})<L(X)$ for some $j \in \{1,\dots,n\}$ and $S_1,S_2 \subseteq \{1,\dots, n \} \backslash \{j\}$, then replace $g_i$ with $g_jg_i$ if $i \in S_1$ and $g_i$ with $g_ig_j^{-1}$ if $i \in S_2$, and return to $(2)$.
\item Return {\tt true} and the $n$-tuple $X=(g_1,\dots,g_n)$.
\end{enumerate}
\end{alg}

\begin{proof}[Proof of Theorem $\ref{thm:alg}$]
It suffices to show that, subject to \Cref{conj}, \Cref{alg} terminates and produces the correct output.

The only recursive step in \Cref{alg} is step $(3)$. Since $L(X)$ is a positive integer which strictly decreases upon each iteration, the algorithm must eventually terminate.

If the algorithm returns false, then $H$ is not both discrete and free by \Cref{elltree}. Otherwise, we must reach step $(4)$ and hence $X$ is minimal. If \Cref{conj} holds, then the elements of $X$ must satisfy the hypotheses of \Cref{PPL} and hence $H$ is both discrete and free of rank $n$.
\end{proof}

\begin{proof}[Proof of Theorem $\ref{ihara}$]
If $H$ contains an elliptic element, then it cannot be both discrete and free by \Cref{elltree}. If $H$ contains no elliptic element, then \Cref{alg} will terminate and return {\tt true}.
\end{proof}

Note that \Cref{alg} also gives a constructive proof of Weidmann's Theorem \cite[Theorem 4]{AG} for locally finite simplicial trees: by recording the product replacements performed in step $(3)$, the elliptic element returned at step $(2)$ can be written as a word in the input generators $h_1, \dots, h_n$.

\section{Implementation of the algorithm}\label{algimp}

\Cref{alg} has exponential complexity in terms of the input size $n$: for each $j$, there are $4^{n-1}$ choices of the subsets $S_1,S_2$ in step $(3)$. For isometries of the Bruhat-Tits tree $T_p$ represented by elements of $\slq$ (viewed as elements of $\slqp$), we have implemented \Cref{alg} in {\sc Magma}. Our implementation is publicly available; see \cite{git}. It runs efficiently when $n, p$ and the translation lengths of the input elements are small.

For a fixed prime $p$, we may generate `random' hyperbolic elements of $\slqp$ as rational matrices of the form 
$\left[ \begin{array}{cc}
ap^{e} & bp^{f}\\ [4pt]
cp^{g} & d \end{array} \right]$
where $a,b,c, e, f,g$ are randomly chosen from the interval $[-N,N] \subseteq \zz$ for some fixed positive integer $N$, and $d \in \qq$ is such that the resulting matrix has determinant 1. Such elements have a translation length of at most $2(3N+\lfloor \log_p(N) \rfloor)$.

In \Cref{runtime}, we record average runtimes for \Cref{alg} across 1000 such $n$-generated subgroups of $\slqp$ for $N=10$. These runtimes were obtained using {\sc Magma V2.25-9}.

\begin{table}[h!]
\begin{tabular}{c|ccccc}
\backslashbox{$p$}{$n$} & 2 & 3 & 4 & 5 & 6  \\
\hline
2  & 0.001 & 0.017 & 0.225 & 1.679 & 9.939 \\ 
3  & 0.001 & 0.027 & 0.411 & 3.773 & 36.681 \\ 
5  & 0.001 & 0.031 & 0.563 & 5.036 & 54.078 \\ 
7 & 0.001 & 0.033 & 0.648 & 6.615  &  62.153 \\ 
11  & 0.001 & 0.040 & 0.748  & 7.838 & 82.230  \\ 
13  & 0.001 & 0.042 & 0.805 & 8.644 & 88.538  \\ 
\end{tabular}
\caption{Average runtime (in seconds) for \Cref{alg} across 1000 `random' $n$-generated subgroups of $\slqp$.} \label{runtime}
\vspace{-0.5cm}
\end{table}

\section{Comments on the conjecture}\label{sec:conj}

There is significant computational evidence that \Cref{conj} holds for $n>3$: using {\sc Magma}, we have generated millions of examples of minimal $n$-tuples of rational matrices of determinant 1 (viewed as elements of $\slqp$, and hence isometries of $T_p$) for $4 \le n \le 10$ and primes $p \le 11$. In every case, the conclusion of \Cref{conj} holds. 

We now discuss some reasons why our definition of minimality cannot be simplified. An algorithm given in \cite{C} takes as input a pair of isometries $(g_1,g_2)$ of a simplicial tree $T$ and performs single product replacements which strictly decrease the sum $l(g_1)+l(g_2)$ until a `minimal' Nielsen-equivalent pair is obtained. Such a `minimal' pair either contains an elliptic isometry or satisfies the hypotheses of the Ping Pong Lemma. The following example shows that this notion of minimality does not immediately generalise to $n$-tuples of isometries.

\begin{exm}\label{notsum}
Consider the following matrices in ${\rm SL_2}(\qq_5)$:
$$
g_1=\left[ \begin{array}{cc}
\frac{1}{5} & -\frac{1}{5}\\ [4pt]
-\frac{1}{5} & \frac{26}{5} \end{array} \right], \;\;
g_2=\left[ \begin{array}{cc}
-1 & 1\\ [4pt]
-\frac{1}{5} & -\frac{4}{5} \end{array} \right], \;\;
g_3=\left[ \begin{array}{cc}
5 & 0\\ [4pt]
0 & \frac{1}{5} \end{array} \right].
$$
Since $v_5({\rm tr}(g_i))=-1$ for each $i$, it follows from \Cref{TL} that each $g_i$ is a hyperbolic isometry of the Bruhat-Tits tree $T_5$ with translation length 2. Observe from \Cref{overlap} that ${\rm Proj}_{\gamma_3}(\gamma_1)\cup {\rm Proj}_{\gamma_3}(\gamma_2)$ is a subpath of $\gamma_3$ of length 2 and hence the 3-tuple $X=(g_1,g_2,g_3)$ does not satisfy the hypotheses of the Ping Pong Lemma. Moreover, no product replacements can strictly reduce the sum $l(g_1)+l(g_2)+l(g_3)$.
\end{exm}

We also present an example which motivated the notion of minimality given in \Cref{min}.

\begin{exm}\label{notsingle}
Let $X=(g_1,g_2,g_3,g_4,g_5)$ be a 5-tuple of hyperbolic isometries of the Bruhat-Tits tree $T_7$, where

$$
g_1=\left[ \begin{array}{cc}
\frac{129}{49} & -\frac{178}{49}\\ [4pt]
\frac{6}{49} & \frac{31}{147} \end{array} \right], \;\;
g_2=\left[ \begin{array}{cc}
-\frac{688}{49} & -\frac{1}{7}\\ [4pt]
\frac{1031}{49} & \frac{1}{7} \end{array} \right], \;\;
g_3=\left[ \begin{array}{cc}
-\frac{1}{49} & -\frac{3}{49}\\ [4pt]
2 & -43 \end{array} \right],
$$
$$
g_4=\left[ \begin{array}{cc}
\frac{9}{7} & -\frac{25}{21}\\ [4pt]
-\frac{60}{49} & \frac{281}{147} \end{array} \right], \;\;
g_5=\left[ \begin{array}{cc}
7 & 7\\ [4pt]
-\frac{3}{7} & -\frac{2}{7} \end{array} \right].
$$

Each isometry is hyperbolic, with $l(g_1)=l(g_2)=l(g_3)=l(g_4)=4$ and $l(g_5)=2$. Of the twenty products of the form $g_ig_j$ or $g_ig_j^{-1}$ for $1 \le i< j \le 5$, ten have translation length 8, four have translation length 6 and six have translation length 4. Thus $L(X)=146$.

Observe from \Cref{overlap} that ${\rm Proj}_{\gamma_5}(\gamma_1)\cup {\rm Proj}_{\gamma_5}(\gamma_2)$ is a subpath of $\gamma_5$ of length 2; it follows that $X$ does not satisfy the hypotheses of the Ping Pong Lemma. There is also no single product replacement which strictly reduces the sum $L(X)$. Moreover, for each $j \in \{1,2,3,4,5\}$, there is no subset $S$ of $\{1, 2, 3, 4, 5\} \backslash \{j\}$ for which replacing every $g_i$ with $i \in S$ by $g_jg_i$ (respectively $g_ig_j^{-1}$) strictly reduces the value of $L(X)$. However, the replacements
$$g_1 \mapsto g_5g_1, \;\; g_3 \mapsto g_5g_3g_5^{-1}$$
produce the 5-tuple $X^{5}_{\{1,3\},\{3\}}$ with $L(X^{5}_{\{1,3\},\{3\}})=144<L(X)$.
\end{exm}

\section{Proof of \Cref{thm:conj}}\label{proofs}

In this final section, we prove \Cref{thm:conj}. We split the proof into several lemmas. The first lemma proves that \Cref{conj} holds when $n=2$. The remaining six lemmas together prove that \Cref{conj} holds when $n=3$; each lemma corresponds to one possible case of how the axes $\gamma_1, \gamma_2, \gamma_3$ can interact.

Throughout the section, we will use $\Delta(\gamma_i,\gamma_j)$ to denote the length of the path of intersection $\gamma_i \cap \gamma_j$ between the axes $\gamma_i$ and $\gamma_j$.

\begin{lem}\label{2}
Let $X=(g_1, g_2)$ be a pair of hyperbolic isometries of a simplicial tree. If $X$ does not satisfy the hypotheses of the Ping Pong Lemma, then it is not minimal.
\end{lem}
\begin{proof}
Since $X$ does not satisfy the hypotheses of the Ping Pong Lemma, the axes $\gamma_1$ and $\gamma_2$ overlap in a path of length $\Delta(\gamma_1,\gamma_2) \ge \min\{l(g_1), l(g_2)\}.$ Without loss of generality, we may further assume that $l(g_1) \le l(g_2)$ and that $\gamma_1$ and $\gamma_2$ have the same orientation. We will show that $L(X^{1}_{\{\},\{2\}})<L(X)$, so $X$ is not minimal.

First observe that replacing $g_2$ by $g_2g_1^{-1}$ preserves $l(g_1)$ and changes the other terms in the expression of $L(X)$ as follows:
\begin{align*}
l(g_1g_2) \mapsto l(g_2) \mapsto l(g_1g_2^{-1}) \mapsto l(g_1^2g_2^{-1}).
\end{align*}
\Cref{overlap} $(2)$ shows that $l(g_1g_2)=l(g_1)+l(g_2)$, so it remains to find $l(g_1^2g_2^{-1})$. Since $g_1^2$ has the same axis as $g_1$ but double the translation length, \Cref{overlap} (applied to $g_1^2$ and $g_2^{-1}$) shows that 
\begin{align*}
l(g_1^2g_2^{-1}) &\le 2l(g_1)+l(g_2)-2\Delta \\
& \le l(g_1g_2)-l(g_1),
\end{align*}
and it follows that $L(X^{1}_{\{\},\{2\}})<L(X)$ as required.
\end{proof}

\begin{lem}\label{3noint}
Let $X=(g_1,g_2,g_3)$ be a triple of hyperbolic isometries of a simplicial tree such that no pair of the axes $\gamma_1, \gamma_2, \gamma_3$ intersects. If $X$ does not satisfy the hypotheses of the Ping Pong Lemma, then it is not minimal.
\end{lem}
\begin{proof}
Without loss of generality, we may assume that the distance $\Delta_3$ between the vertices ${\rm Proj}_{\gamma_3}(\gamma_1)$ and ${\rm Proj}_{\gamma_3}(\gamma_2)$ is at least $l(g_3)$, and that $g_3$ translates $\gamma_1$ towards $\gamma_2$. Consider the replacement $g_1 \mapsto g_3g_1g_3^{-1}$. This preserves all terms in the expression of $L(X)$ except for the following:
\begin{align}\label{1to313}
l(g_1g_2) & \mapsto l(g_3g_1g_3^{-1}g_2) \\
l(g_1g_2^{-1}) & \mapsto l(g_3g_1g_3^{-1}g_2^{-1}). \notag 
\end{align}

The axis of $g_3g_1g_3^{-1}$ is $g_3 \cdot \gamma_1$, which is at distance at most $d(\gamma_1, \gamma_2)-l(g_3)$ from $\gamma_1$. Hence \Cref{overlap} shows that 
\begin{align*}
l(g_3g_1g_3^{-1}g_2) &\le l(g_3g_1g_3^{-1})+l(g_2)+2(d(\gamma_1, \gamma_2)-l(g_3)) \\
& = l(g_1g_2)-2l(g_3)
\end{align*}
and, similarly, $l(g_3g_1g_3^{-1}g_2^{-1}) \le l(g_1g_2^{-1})-2l(g_3)$. Thus $L(X^{3}_{\{1\},\{1\}})<L(X)$ and $X$ is not minimal.
\end{proof}

\begin{lem}\label{3onepair}
Let $X=(g_1,g_2,g_3)$ be a triple of hyperbolic isometries of a simplicial tree such that exactly one pair of the axes $\gamma_1, \gamma_2, \gamma_3$ intersects. If $X$ does not satisfy the hypotheses of the Ping Pong Lemma, then it is not minimal.
\end{lem}
\begin{proof}
Without loss of generality, we suppose that $\gamma_1$ and $\gamma_3$ are the only axes which intersect, and that they intersect with the same orientation. We consider two cases, depending on whether or not the vertices ${\rm Proj}_{\gamma_1}(\gamma_2)$ and ${\rm Proj}_{\gamma_3}(\gamma_2)$ coincide.

In the first case, suppose that ${\rm Proj}_{\gamma_1}(\gamma_2) \neq {\rm Proj}_{\gamma_3}(\gamma_2)$. We may therefore assume (after interchanging the roles of $g_1$ and $g_3$, if necessary) that ${\rm Proj}_{\gamma_3}(\gamma_2)$ does not lie on $\gamma_1 \cap \gamma_3$, and that $g_3$ translates $\gamma_1$ towards $\gamma_2$. Let $\Delta_3$ denote the length of the shortest subpath of $\gamma_3$ containing both ${\rm Proj}_{\gamma_3}(\gamma_1)=\gamma_1 \cap \gamma_3$ and ${\rm Proj}_{\gamma_3}(\gamma_2)$. Since $X$ does not satisfy the hypotheses of the Ping Pong Lemma, either $\Delta_3 \ge l(g_3)$, or $\Delta_3 < l(g_3)$ and $\Delta(\gamma_1,\gamma_3)\ge l(g_1)$.

If $\Delta_3 \ge l(g_3)$, then consider the replacement $g_1 \mapsto g_3g_1g_3^{-1}$, which changes the terms of $L(X)$ as in (\ref{1to313}). Since the axis $g_3\cdot \gamma_1$ is strictly closer to $\gamma_2$ than $\gamma_1$, this shows (as in the proof of \Cref{3noint}) that $X$ is not minimal.

If $\Delta_3 < l(g_3)$ and $\Delta(\gamma_1,\gamma_3)\ge l(g_1)$, then consider the product replacement $g_3 \mapsto g_3g_1^{-1}$. This preserves all terms in the expression of $L(X)$ except for the following:
\begin{align}\label{3to31}
l(g_1g_3) &\mapsto l(g_3) \mapsto l(g_1g_3^{-1}) \mapsto l(g_1^2g_3^{-1})  \notag \\
l(g_2g_3)  &\mapsto l(g_1^{-1}g_2g_3)  \\
l(g_2g_3^{-1})  &\mapsto l(g_1^{-1}g_2^{-1}g_3). \notag 
\end{align}

Since $g_1^2$ has the same axis as $g_1$ but twice the translation length, 
\begin{align*}
l(g_1^2g_3^{-1}) & \le l(g_1g_3)+l(g_1)-2\Delta(\gamma_1,\gamma_3) \\
& \le l(g_1g_3)-l(g_1)
\end{align*}
by \Cref{overlap}. The axes of $g_2g_3$ and $g_2^{-1}g_3$ both intersect the axis of $g_1^{-1}$ with opposite orientations along a subpath of length $\Delta(\gamma_1,\gamma_3) \ge l(g_1)$; see the upper diagram of \Cref{case2}. Hence \Cref{overlap} also shows that $l(g_1^{-1}g_2g_3) \le l(g_2g_3)-l(g_1)$ and $l(g_1^{-1}g_2^{-1}g_3) \le l(g_2g_3^{-1})-l(g_1)$. Thus $L(X^{1}_{\{\},\{3\}})<L(X)$ and $X$ is not minimal.

In the second case, suppose that ${\rm Proj}_{\gamma_1}(\gamma_2)={\rm Proj}_{\gamma_3}(\gamma_2)$. This vertex bisects $\gamma_1 \cap \gamma_3$ into two subpaths, and we may assume that $g_1$ and $g_3$ both translate the longer of these subpaths (of length which we will denote by $\delta_3$) towards the shorter one. Without loss of generality, we may assume that $l(g_1) \le l(g_3)$. Since $X$ does not satisfy the hypotheses of the Ping Pong Lemma, $\Delta(\gamma_1,\gamma_3) \ge l(g_1)$ and hence $\delta_3 \ge \frac{l(g_1)}{2}$.

Consider the product replacement $g_3 \mapsto g_3g_1^{-1}$, which changes the terms of $L(X)$ as in (\ref{3to31}). As in the previous case, $l(g_1^2g_3^{-1}) \le l(g_1g_3)-l(g_1)$. If $l(g_3)>\delta_3$, then the axes of $g_2g_3$ and $g_2^{-1}g_3$ both intersect the axis of $g_1^{-1}$ with opposite orientations along a subpath of length $\delta_3$; see the lower diagram of \Cref{case2}. 

\begin{figure}[h]

\centering
\begin{tikzpicture}
  [scale=0.8,auto=left] 
  
  \draw[dashed,>->] (8,0) to (18,0); \node at (7.5,0) {$\gamma_3$}; \node at (18.5,0) {$\gamma_3$};
\draw[dashed,->] (14,0) to (14,3); \node at (14,3.5) {$\gamma_1$}; \draw[dashed,>-] (11,3) to (11,0); \node at (11,3.5) {$\gamma_1$};
\draw[dashed,>->] (19,2) to (16,2) to (16,5);  \node at (16,5.5) {$\gamma_2$}; \node at (19.5,2) {$\gamma_2$};
 \node[circle,inner sep=0pt,minimum size=3,fill=black] (1) at (11, 0) {}; 
  \node[circle,inner sep=0pt,minimum size=3,fill=black] (1) at (14, 0) {}; \node at (15.7,0.3) {$x$};
\node[circle,inner sep=0pt,minimum size=3,fill=black] (1) at (16, 0) {}; \draw[dashed] (16,0) to (16,2);
 \node[circle,inner sep=0pt,minimum size=3,fill=black] (1) at (16, 2) {}; \node at (15.7,1.95) {$y$};
\draw[dotted,|-|] (11,-0.3) to (14,-0.3); \node at (12.5,-0.8) {$\Delta(\gamma_1,\gamma_3) \ge l(g_1)$};
\draw[dotted,|-|] (11,-1.3) to (16,-1.3); \node at (13.5,-1.7) {$\Delta_3<l(g_3)$};

 \node[circle,inner sep=0pt,minimum size=3,fill=black] (1) at (10, 0) {}; \node at (10,0.5) {$g_3^{-1}x$};
    \draw[dashed, >-] (8.5,-3) to (8.5,-1.5) to (10,0)   ; \node at (8.5,-3.5) {${\rm Axis}(g_2^{-1}g_3)$};
        \draw[dashed, >-] (7,-1.5) to (8.5,-1.5) ; \node at (5.6,-1.5) {${\rm Axis}(g_2g_3)$};
  \node[circle,inner sep=0pt,minimum size=3,fill=black] (1) at (8.5,-1.5) {}; \node at (9.2,-1.5) {$g_3^{-1}y$};

 \node[circle,inner sep=0pt,minimum size=3,fill=black] (1) at (17, 2) {}; \node at (17,1.6) {$g_2^{-1}y$};
  \draw[dashed, ->] (17,2) to (19,4); \node at (20.4,4.3) {${\rm Axis}(g_2^{-1}g_3)$};
   \node[circle,inner sep=0pt,minimum size=3,fill=black] (1) at (18,3) {}; \node at (18.7,3) {$g_2^{-1}x$};  
   
    \node[circle,inner sep=0pt,minimum size=3,fill=black] (1) at (16, 3) {}; \node at (15.5,3) {$g_2y$};
  \draw[dashed, ->] (16,3) to (18,5); \node at (19.4,5.3) {${\rm Axis}(g_2g_3)$};
   \node[circle,inner sep=0pt,minimum size=3,fill=black] (1) at (17,4) {}; \node at (17.5,3.9) {$g_2x$};  

\end{tikzpicture}

\begin{tikzpicture}
  [scale=0.8,auto=left] 

\draw[dashed,>->] (-1,0) to (8,0); \node at (-1.5,0) {$\gamma_3$}; \node at (8.5,0) {$\gamma_3$};
\draw[dashed,->] (6,0) to (6,2.5); \node at (6,3) {$\gamma_1$}; \draw[dashed,>-] (2,2.5) to (2,0); \node at (2,3) {$\gamma_1$};
\draw[dashed,>->] (6.5,5) to (4.5,3) to (2.5,5);  \node at (2.2,5.3) {$\gamma_2$}; \node at (6.8,5.3) {$\gamma_2$};
 \node[circle,inner sep=0pt,minimum size=3,fill=black] (1) at (2, 0) {}; 
  \node[circle,inner sep=0pt,minimum size=3,fill=black] (1) at (4.5, 0) {}; \draw[dashed] (4.5,0) to (4.5,3); \node at (4.8,0.3) {$x$};
\node[circle,inner sep=0pt,minimum size=3,fill=black] (1) at (6, 0) {}; 
 \node[circle,inner sep=0pt,minimum size=3,fill=black] (1) at (4.5, 3) {}; \node at (4.8,2.9) {$y$};
\draw[dotted,|-|] (2,-0.3) to (6,-0.3); \node at (4,-0.8) {$\Delta(\gamma_1,\gamma_3)\ge l(g_1)$};
\draw[dotted,|-|] (2,-1.3) to (4.5,-1.3); \node at (3.3,-1.8) {$\delta_3<l(g_3)$};

 \node[circle,inner sep=0pt,minimum size=3,fill=black] (1) at (1, 0) {}; \node at (1,0.5) {$g_3^{-1}x$};
    \draw[dashed, >-] (-0.5,-3) to (-0.5,-1.5) to (1,0); \node at (-0.5,-3.5) {${\rm Axis}(g_2^{-1}g_3)$};
        \draw[dashed, >-]  (-2,-1.5) to (-0.5,-1.5); \node at (-3.4,-1.5) {${\rm Axis}(g_2g_3)$};
  \node[circle,inner sep=0pt,minimum size=3,fill=black] (1) at (-0.5,-1.5) {}; \node at (0.2,-1.5) {$g_3^{-1}y$};
   \node[circle,inner sep=0pt,minimum size=3,fill=black] (1) at (5.5, 4) {}; \node at (6.2, 4) {$g_2^{-1}y$};
    \node[circle,inner sep=0pt,minimum size=3,fill=black] (1) at (5.5, 6) {}; \node at (6.2, 6) {$g_2^{-1}x$};
   \draw[dashed,->]  (5.5,4) to (5.5,7); \node at (5.75,7.5) {${\rm Axis}(g_2^{-1}g_3)$};
   
      \node[circle,inner sep=0pt,minimum size=3,fill=black] (1) at (3.5, 4) {}; \node at (2.9, 3.9) {$g_2y$};
    \node[circle,inner sep=0pt,minimum size=3,fill=black] (1) at (3.5, 6) {}; \node at (2.9, 6) {$g_2x$};
   \draw[dashed,->] (3.5,4) to (3.5,7); \node at (3.25,7.5) {${\rm Axis}(g_2g_3)$};

\end{tikzpicture}

\caption{Exactly one pair of the axes $\gamma_1,\gamma_2$ and $\gamma_3$ intersects.} \label{case2}
\end{figure}
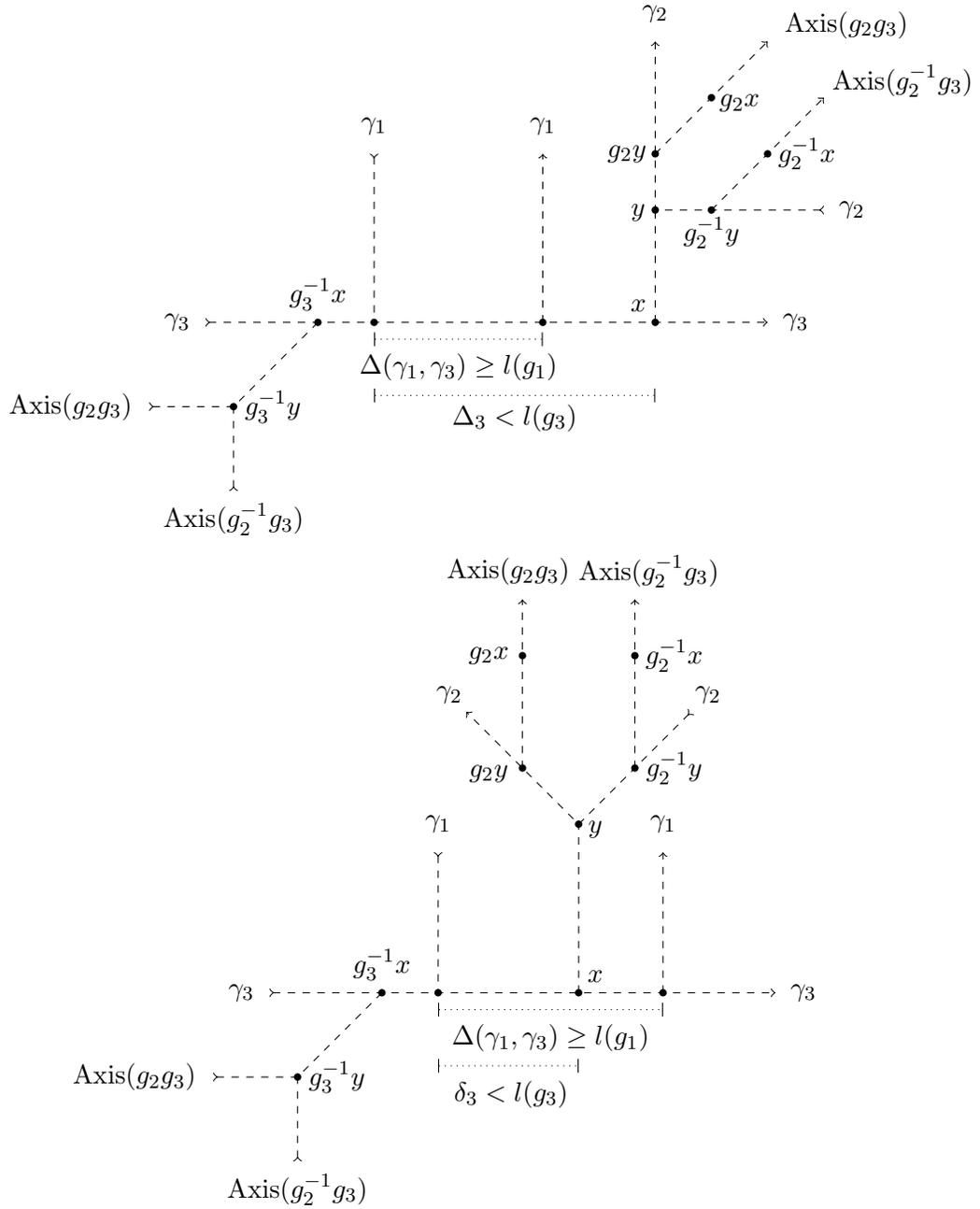

\Cref{overlap} $(3)$ shows that
\begin{align*}
l(g_1^{-1}g_2g_3) &\le l(g_2g_3)+l(g_1)-2\delta_3 \\
& \le l(g_2g_3)
\end{align*}
and, similarly, $l(g_1^{-1}g_2^{-1}g_3) \le l(g_2g_3^{-1})$. Hence $L(X^{1}_{\{\},\{3\}})<L(X)$ and $X$ is not minimal. 

On the other hand, if $l(g_3) \le \delta_3$, then the axes of $g_2g_3$ and $g_2^{-1}g_3$ both intersect the axis of $g_1^{-1}$ with opposite orientations along a subpath of length at least $l(g_3) \ge l(g_1)$. \Cref{overlap} shows that $l(g_1^{-1}g_2g_3) \le l(g_2g_3)-l(g_1)$ and $l(g_1^{-1}g_2^{-1}g_3)  \le l(g_2g_3^{-1})-l(g_1)$, so again $X$ is not minimal.
\end{proof}

\begin{lem}\label{3twopairs}
Let $X=(g_1,g_2,g_3)$ be a triple of hyperbolic isometries of a simplicial tree such that exactly two pairs of the axes $\gamma_1, \gamma_2, \gamma_3$ intersect. If $X$ does not satisfy the hypotheses of the Ping Pong Lemma, then it is not minimal.
\end{lem}
\begin{proof}
Without loss of generality, we may assume that $\gamma_3$ intersects both $\gamma_1$ and $\gamma_2$ with the same orientation, and that $g_3$ translates $\gamma_1$ towards $\gamma_2$. Let $\Delta_3$ denote the length of the shortest subpath of $\gamma_3$ containing both ${\rm Proj}_{\gamma_3}(\gamma_1)=\gamma_1 \cap \gamma_3$ and ${\rm Proj}_{\gamma_3}(\gamma_2)=\gamma_2 \cap \gamma_3$. 

If $l(g_3)<\Delta_3+d(\gamma_1,\gamma_2)$, then consider the replacement $g_1 \mapsto g_3g_1g_3^{-1}$, which changes the terms of $L(X)$ as in (\ref{1to313}). Since the axis $g_3\cdot \gamma_1$ is strictly closer to $\gamma_2$ than $\gamma_1$ (and possibly even intersects $\gamma_2$), it follows as in the proof of \Cref{3noint} that $X$ is not minimal. Hence we may suppose that $l(g_3) \ge \Delta_3+d(\gamma_1,\gamma_2)$. In particular, $\Delta_3<l(g_3)$ and, since $X$ does not satisfy the hypotheses of the Ping Pong Lemma, we may additionally suppose (after interchanging the roles of $g_1$ and $g_2$, if necessary) that $\Delta(\gamma_1,\gamma_3) \ge l(g_1)$.

We first consider the product replacement $g_3 \mapsto g_3g_1^{-1}$, which changes the terms of $L(X)$ as in (\ref{3to31}). As before, $l(g_1^2g_3^{-1}) \le l(g_1g_3)-l(g_1)$. The axis of $g_2g_3$ intersects the axis of $g_1^{-1}$ with opposite orientations along a subpath of length $\Delta(\gamma_1,\gamma_3)$; see the upper diagram of \Cref{case3}. \Cref{overlap} hence shows that 
\begin{align*}
l(g_1^{-1}g_2g_3) &= l(g_2g_3)+l(g_1)-2\Delta(\gamma_1,\gamma_3) \\
&\le l(g_2g_3)-l(g_1).
\end{align*}

If $\Delta(\gamma_2,\gamma_3) < l(g_2)$, then the axis of $g_2^{-1}g_3$ intersects $\gamma_1$; see the upper diagram of \Cref{case3}. Similarly, if $\Delta(\gamma_2,\gamma_3) > l(g_2)$, then the axis of $g_2^{-1}g_3$ also intersects $\gamma_1$; see \Cref{bigint}. In either case, \Cref{overlap} shows that $l(g_1^{-1}g_2^{-1}g_3) \le l(g_2g_3^{-1})+l(g_1)$ and hence $L(X^{1}_{\{\},\{3\}})<L(X)$, so $X$ is not minimal. Thus we may assume that $\Delta(\gamma_2,\gamma_3) = l(g_2)$.

By \Cref{overlap}, $l(g_2^{-1}g_3)=\max\{0,l(g_3)-l(g_2)-2\Delta'\}$, where $\Delta'$ denotes the length of the path $\gamma_3 \cap g_2 \cdot \gamma_3$. Consider the product replacements $g_1 \mapsto g_1g_2^{-1}$ and $g_3 \mapsto g_3g_2^{-1}$, which preserve all terms in the expression of $L(X)$ except for the following:
\begin{align}\label{1to12and3to32}
l(g_1g_2) &\mapsto l(g_1) \mapsto l(g_1g_2^{-1}) \mapsto l(g_1g_2^{-2})  \notag \\
l(g_2g_3) &\mapsto l(g_3) \mapsto l(g_2g_3^{-1}) \mapsto l(g_2^2g_3^{-1})   \\
l(g_1g_3)  &\mapsto l(g_1g_2^{-1}g_3g_2^{-1})=l(g_2^{-1}g_3g_2^{-1}g_1). \notag 
\end{align}
Since $g_2^2$ has the same axis as $g_2$ but twice the translation length, \Cref{overlap} shows that $l(g_1g_2^{-2})=l(g_1g_2)+l(g_2)$ and $l(g_2^2g_3^{-1}) \le l(g_2g_3)-l(g_2).$

\begin{figure}[h!]
\centering

\begin{tikzpicture}
  [scale=0.8,auto=left] 

\draw[dashed,>->] (-2.5,0) to (10,0); \node at (-3,0) {$\gamma_3$}; \node at (10.5,0) {$\gamma_3$};
\draw[dashed,->] (4,0) to (4,2.5); \node at (4,3) {$\gamma_1$}; \draw[dashed,>-] (2,2.5) to (2,0); \node at (2,3) {$\gamma_1$};
\draw[dashed,->] (8,0) to (8,-4.5); \node at (8,-5) {$\gamma_2$}; \draw[dashed,>-] (6,-4.5) to (6,0); \node at (6,-5) {$\gamma_2$};
 \node[circle,inner sep=0pt,minimum size=3,fill=black] (1) at (2, 0) {}; 
  \node[circle,inner sep=0pt,minimum size=3,fill=black] (1) at (4, 0) {}; \node at (8,0.4) {$q$};
\node[circle,inner sep=0pt,minimum size=3,fill=black] (1) at (6, 0) {}; \node at (6,0.4) {$p$};
\node[circle,inner sep=0pt,minimum size=3,fill=black] (1) at (8, 0) {};
\draw[dotted,|-|] (2,-1.7) to (4,-1.7); \node at (3,-2.3) {$\Delta(\gamma_1,\gamma_3) \ge l(g_1)$};
\draw[dotted,|-|] (4.1,1) to (5.9,1); \node at (5,1.5) {$d(\gamma_1,\gamma_2)$};
\draw[dotted,|-|] (6,1) to (8,1); \node at (7.9,1.5) {$\Delta(\gamma_2,\gamma_3) < l(g_2)$};
\draw[dotted,|-|] (2,-0.5) to (7.9,-0.5); \node at (4.5,-1) {$\Delta_3<l(g_3)$};

 \node[circle,inner sep=0pt,minimum size=3,fill=black] (1) at (1, 0) {}; \node at (1,0.5) {$g_3^{-1}q$};
  \node[circle,inner sep=0pt,minimum size=3,fill=black] (1) at (-1, 0) {}; \node at (-1,-0.5) {$g_3^{-1}p$};
  
 \node[circle,inner sep=0pt,minimum size=3,fill=black] (1) at (8, -3) {}; \node at (7.5,-3) {$g_2q$};
  \node[circle,inner sep=0pt,minimum size=3,fill=black] (1) at (6, -1.5) {}; \node at (6.6,-1.5) {$g_2^{-1}q$};

 \draw[dashed, ->] (8,-3) to (9.5,-4.5); \node at (10.7,-4.8) {${\rm Axis}(g_2g_3)$};
 \draw[dashed, >-] (-2.5,1.5) to (-1,0); \node at (-3.7,1.8) {${\rm Axis}(g_2g_3)$};

  \draw[dashed, ->] (6,-1.5) to (4.5,-3); \node at (4,-3.5) {${\rm Axis}(g_2^{-1}g_3)$};  
  \draw[dashed, >-] (-0.5,-1.5) to (1,0); \node at (-1.7,-1.8) {${\rm Axis}(g_2^{-1}g_3)$};

\end{tikzpicture}

\begin{tikzpicture}
  [scale=0.8,auto=left] 

\draw[dashed,>->] (-3,0) to (9.5,0); \node at (-3.5,0) {$\gamma_3$}; \node at (10,0) {$\gamma_3$};
\draw[dashed,->] (4,0) to (4,3); \node at (4,3.5) {$\gamma_1$}; \draw[dashed,>-] (1,3) to (1,0); \node at (1,3.5) {$\gamma_1$};
\draw[dashed,->] (8,0) to (8,-3.5); \node at (8,-4) {$\gamma_2$}; \draw[dashed,>-] (6,-3.5) to (6,0); \node at (6,-4) {$\gamma_2$};
 \node[circle,inner sep=0pt,minimum size=3,fill=black] (1) at (1, 0) {}; 
    \node[circle,inner sep=0pt,minimum size=3,fill=black] (1) at (4, 0) {}; \node at (4.3,0.3) {$x$};
\node[circle,inner sep=0pt,minimum size=3,fill=black] (1) at (6, 0) {}; \node at (6,0.4) {$p$};
\node[circle,inner sep=0pt,minimum size=3,fill=black] (1) at (8, 0) {}; \node at (8,0.4) {$q$};
\draw[dotted,|-|] (1,-2.6) to (4,-2.6); \node at (2.5,-3.2) {$\Delta(\gamma_1,\gamma_3) \ge l(g_1)$};
\draw[dotted,|-|] (-0.5,-0.5) to (5.9,-0.5); \node at (2.75,-1) {$\Delta'>\Delta(\gamma_1,\gamma_3)+d(\gamma_1,\gamma_2)$};
\draw[dotted,|-|] (6.1,1) to (8,1); \node at (7.9,1.5) {$\Delta(\gamma_2,\gamma_3) = l(g_2)$};
\draw[dotted,|-|] (4.1,1) to (5.9,1); \node at (5,1.5) {$d(\gamma_1,\gamma_2)$};
\draw[dotted,|-|] (1,-4.5) to (8,-4.5); \node at (4.5,-5) {$\Delta_3<l(g_3)$};

   \node[circle,inner sep=0pt,minimum size=3,fill=black] (1) at (-0.5, 0) {}; \node at (-0.5,0.3) {$r$}; 
     \node[circle,inner sep=0pt,minimum size=3,fill=black] (1) at (-1.5, 0) {}; \node at (-2,-0.4) {$g_3^{-1}g_2r$}; 
 
  \node[circle,inner sep=0pt,minimum size=3,fill=black] (1) at (2, 0) {}; \node at (2.7,0.5) {$g_1^{-1}x$};     
        \node[circle,inner sep=0pt,minimum size=3,fill=black] (1) at (6, -2) {}; \node at (6.7,-2) {$g_2^{-1}p$};

    \draw[dashed, >-] (2,2) to (2,0); \node at (2.5,2.5) {${\rm Axis}(g_2^{-1}g_1)$}; 
        \draw[dashed, ->] (6,-2) to (4.5,-3.5); \node at (4,-4) {${\rm Axis}(g_2^{-1}g_1)$}; 
  
    \draw[dashed, >-] (-3,1.5) to (-1.5,0); \node at (-3.5,2) {${\rm Axis}(g_2^{-1}g_3)$}; 
        \draw[dashed, ->] (-0.5,0) to (-2,-1.5); \node at (-2.5,-2) {${\rm Axis}(g_2^{-1}g_3)$};

\end{tikzpicture}

\caption{Exactly two pairs of the axes $\gamma_1, \gamma_2$ and $\gamma_3$ intersect.} \label{case3}
\end{figure}

Let us first suppose that $g_2^{-1}g_3$ is hyperbolic. If the axis of $g_2^{-1}g_3$ intersects $\gamma_1$, then $X$ is not minimal by the above argument. So suppose that the axis of $g_2^{-1}g_3$ does not intersect $\gamma_1$. If $\Delta' \le \Delta(\gamma_1,\gamma_3)+d(\gamma_1,\gamma_2)$, then $l(g_3)-l(g_2)+\Delta'<d(\gamma_1,\gamma_2)$; see \Cref{bigint3}. This contradicts that $l(g_3) \ge \Delta_3+d(\gamma_1,\gamma_2)$, so we may suppose that $\Delta'>\Delta(\gamma_1,\gamma_3)+d(\gamma_1,\gamma_2)$; see the lower diagram of \Cref{case3}.

Observe that the axis of $g_2^{-1}g_3$ is at distance at most $\Delta'-l(g_1)-d(\gamma_1,\gamma_2)$ from the axis of $g_2^{-1}g_1$; see the lower diagram of \Cref{case3}. \Cref{overlap} shows that
\begin{align*}
l(g_2^{-1}g_3g_2^{-1}g_1)&\le l(g_2^{-1}g_3)+l(g_2^{-1}g_1)+2(\Delta'-l(g_1)-d(\gamma_1,\gamma_2)) \\
&=l(g_1g_3)-2l(g_1)
\end{align*}
and hence $L(X^2_{\{\},\{1,3\}})<L(X)$. We conclude that $X$ is not minimal if $g_2^{-1}g_3$ is hyperbolic.

On the other hand, if $g_2^{-1}g_3$ is elliptic, then \Cref{ellfix} shows that $g_2^{-1}g_3$ fixes a vertex $v$ of $\gamma_3$ at distance $\frac{l(g_3)-l(g_2)}{2}$ from the initial (with respect to translation direction) vertex $p$ of $\gamma_2 \cap \gamma_3$. If $v$ lies on the axis of $g_2^{-1}g_1$, then \cite[Proposition 1.7]{P} shows that
\begin{align*}
l(g_2^{-1}g_3g_2^{-1}g_1)&\le l(g_2^{-1}g_1) \\
&=l(g_1g_3)-l(g_3)+l(g_2)+2d(\gamma_1,\gamma_2) \\
&\le l(g_1g_3)-\Delta(\gamma_1,\gamma_3).
\end{align*}
Otherwise $v$ is at distance $\frac{l(g_3)-l(g_2)}{2}-l(g_1)-d(\gamma_1,\gamma_2)$ from the axis of $g_2^{-1}g_1$ and \cite[Proposition 1.7]{P} shows that
\begin{align*}
l(g_2^{-1}g_3g_2^{-1}g_1)&\le l(g_2^{-1}g_1)+l(g_3)-l(g_2)-2l(g_1)-2d(\gamma_1,\gamma_2) \\
&=l(g_1g_3)-2l(g_1).
\end{align*}
In either case, this shows that $X$ is not minimal.
\end{proof}

\begin{lem}\label{3threepairs1}
Let $X=(g_1,g_2,g_3)$ be a triple of hyperbolic isometries of a simplicial tree such that $\gamma_1 \cap \gamma_2, \gamma_1 \cap \gamma_3$ and $\gamma_2 \cap \gamma_3$ are non-empty, and at least two of these paths are equal. If $X$ does not satisfy the hypotheses of the Ping Pong Lemma, then it is not minimal.
\end{lem}
\begin{proof}
Without loss of generality, suppose that $\gamma_1 \cap \gamma_2= \gamma_2 \cap \gamma_3$. We may further suppose that all three axes $\gamma_1,\gamma_2,\gamma_3$ intersect with the same orientation. Note that $(\gamma_1 \cap \gamma_3) \backslash (\gamma_2 \cap \gamma_3)$ is the disjoint union of two (possibly empty) subpaths and we may assume that $g_3$ translates the longer of these subpaths (of length which we will denote by $\delta_3$) towards the shorter one. Let $\Delta=\Delta(\gamma_1,\gamma_2)=\Delta(\gamma_2,\gamma_3)$, so that $\Delta(\gamma_1,\gamma_3) \le 2\delta_3+\Delta$. Suppose without loss of generality that $l(g_1) \le l(g_3)$. Since $(g_1,g_2,g_3)$ do not satisfy the hypotheses of the Ping Pong Lemma, either $\Delta(\gamma_1,\gamma_3) \ge l(g_1)$, or otherwise $\Delta(\gamma_1,\gamma_3) < l(g_1)$ and $\Delta \ge l(g_2)$.

Let us first suppose that $\Delta(\gamma_1,\gamma_3) \ge l(g_1)$. Consider the product replacement $g_3 \mapsto g_3g_1^{-1}$, which changes the terms of $L(X)$ as in (\ref{3to31}). As before, $l(g_1^2g_3^{-1})\le l(g_1g_3)-l(g_1)$. Observe that the axis of $g_2g_3$ intersects the axis of $g_1^{-1}$ with opposite orientations along a subpath of length either $\Delta+\delta_3$ (if $l(g_3)>\delta_3$) or at least $\Delta+l(g_3)$ (otherwise); see the upper diagram of \Cref{case4} for the former case. If $l(g_3)>\delta_3$, then \Cref{overlap} shows that
\begin{align*}
l(g_1^{-1}g_2g_3) &\le l(g_2g_3)+l(g_1)-2(\Delta+\delta_3) \\
&\le l(g_2g_3)-\Delta.
\end{align*}
On the other hand, if $l(g_3) \le \delta_3$, then $l(g_1^{-1}g_2g_3)\le l(g_2g_3)-2\Delta-l(g_3)$ by \Cref{overlap}, and so it remains to compare $l(g_1^{-1}g_2^{-1}g_3)$ and $l(g_2g_3^{-1})$.

\begin{figure}[h]
\centering

\begin{tikzpicture}
  [scale=0.8,auto=left]

\draw[dashed,>->] (-1,0) to (10,0); \node at (-1.5,0) {$\gamma_3$}; \node at (10.5,0) {$\gamma_3$};
\draw[dashed,>-] (4.5,-3.5) to (4.5,0); \node at (4.5,-4) {$\gamma_2$}; \draw[dashed,>-] (2,3) to (2,0); \node at (2,3.5) {$\gamma_1$};
\draw[dashed,->] (6.5,0) to (6.5,-3.5); \node at (6.5,-4) {$\gamma_2$}; \draw[dashed,->] (8,0) to (8,3); \node at (8,3.5) {$\gamma_1$};
 \node[circle,inner sep=0pt,minimum size=3,fill=black] (1) at (2, 0) {}; 
  \node[circle,inner sep=0pt,minimum size=3,fill=black] (1) at (4.5, 0) {};\node at (4.5,0.4) {$p$};
\node[circle,inner sep=0pt,minimum size=3,fill=black] (1) at (6.5, 0) {};\node at (6.5,0.4) {$q$};
\node[circle,inner sep=0pt,minimum size=3,fill=black] (1) at (8, 0) {};
\draw[dotted,|-|] (4.6,1) to (6.4,1); \node at (5.5,1.5) {$\Delta$};
\draw[dotted,|-|] (2.1,2) to (7.9,2); \node at (5,2.5) {$\Delta(\gamma_1,\gamma_3) \ge l(g_1)$};
\draw[dotted,|-|] (2,-0.5) to (4.4,-0.5); \node at (3.2,-1) {$\delta_3<l(g_3)$};

    \node[circle,inner sep=0pt,minimum size=3,fill=black] (1) at (1, 0) {}; \node at (1,-0.5) {$g_3^{-1}p$};
 \node[circle,inner sep=0pt,minimum size=3,fill=black] (1) at (6.5,-2) {}; \node at (5.9,-2) {$g_2q$};
 \draw[dashed, ->] (6.5,-2) to (8,-3.5); \node at (9.3,-3.9) {${\rm Axis}(g_2g_3)$};
    \draw[dashed, >-] (-0.5,1.5) to (1,0); \node at (-1,2) {${\rm Axis}(g_2g_3)$};

\end{tikzpicture}

\begin{tikzpicture}
  [scale=0.8,auto=left]

\draw[dashed,>->] (-2,0) to (10,0); \node at (-2.5,0) {$\gamma_3$}; \node at (10.5,0) {$\gamma_3$};
\draw[dashed,>-] (1.5,-3) to (4.5,0); \node at (1,-3.5) {$\gamma_2$}; \draw[dashed,>-] (2,3) to (2,0); \node at (2,3.5) {$\gamma_1$};
\draw[dashed,->] (6.5,0) to (9.5,-3); \node at (10,-3.5) {$\gamma_2$}; \draw[dashed,->] (8,0) to (8,3); \node at (8,3.5) {$\gamma_1$};
 \node[circle,inner sep=0pt,minimum size=3,fill=black] (1) at (2, 0) {}; 
  \node[circle,inner sep=0pt,minimum size=3,fill=black] (1) at (4.5, 0) {}; \node at (4.5,0.3) {$p$};
\node[circle,inner sep=0pt,minimum size=3,fill=black] (1) at (6.5, 0) {};\node at (6.5,0.3) {$q$};
\node[circle,inner sep=0pt,minimum size=3,fill=black] (1) at (8, 0) {};
\draw[dotted,|-|] (4.6,0.8) to (6.4,0.8); \node at (5.5,1.3) {$\Delta=l(g_2)>l(g_1)$};
\draw[dotted,|-|] (2.1,2) to (7.9,2); \node at (5,2.5) {$\Delta(\gamma_1,\gamma_3) \ge l(g_1)$};
\draw[dotted,|-|] (2,-0.5) to (4.4,-0.5); \node at (2.7,-1) {$\delta_3$};
\draw[dotted,|-|] (1,-2) to (4.4,-2); \node at (2.75,-2.5) {$\Delta'$};

       \node[circle,inner sep=0pt,minimum size=3,fill=black] (1) at (1, 0) {}; \node at (1,0.3) {$r$}; 
     \node[circle,inner sep=0pt,minimum size=3,fill=black] (1) at (0, 0) {}; \node at (-0.5,-0.4) {$g_3^{-1}g_2r$}; 
         \draw[dashed, >-] (-1.5,1.5) to (0,0); \node at (-2,2) {${\rm Axis}(g_2^{-1}g_3)$}; 
        \draw[dashed, ->] (1,0) to (-0.5,-1.5); \node at (-1,-2) {${\rm Axis}(g_2^{-1}g_3)$}; 

 \node[circle,inner sep=0pt,minimum size=3,fill=black] (1) at (5.5, 0) {}; \node at (5.5,0.4) {$g_1^{-1}q$};
          \draw[dashed, >-] (7,-1.5) to (5.5,0); \node at (7,-2) {${\rm Axis}(g_2^{-1}g_1)$};
          
           \node[circle,inner sep=0pt,minimum size=3,fill=black] (1) at (3.25, -1.25) {}; \node at (4.25,-1.25) {$g_2^{-1}g_1p$};
          \draw[dashed, ->] (3.25,-1.25) to (3.25,-3.25); \node at (3.25,-3.75) {${\rm Axis}(g_2^{-1}g_1)$};

\end{tikzpicture}

\caption{All three pairs of axes $\gamma_1, \gamma_2, \gamma_3$ intersect, and at least two pairs intersect along the same path.} 
\label{case4}

\end{figure}
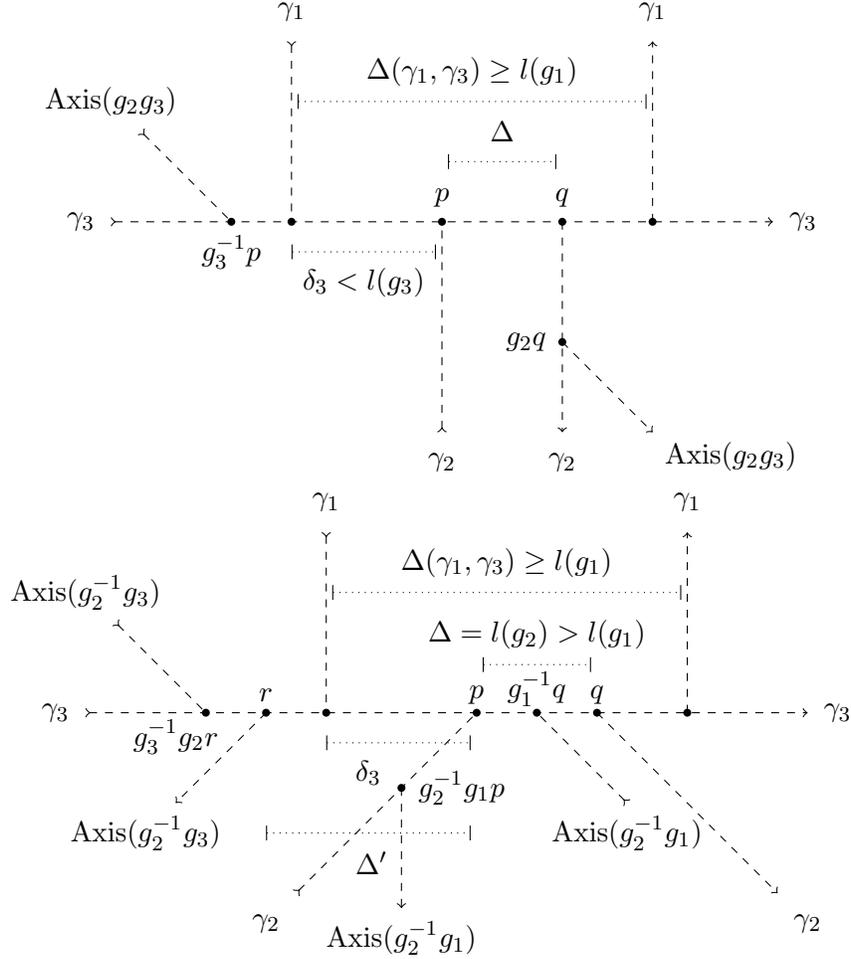

If $\Delta < \min\{l(g_2),l(g_3)\}$, then the axis of $g_2^{-1}g_3$ contains the initial vertex $p$ of $\gamma_2 \cap \gamma_3=\gamma_1\cap \gamma_2$; see \Cref{smallint}. Therefore the axes of $g_1^{-1}$ and $g_2^{-1}g_3$ intersect and $l(g_1^{-1}g_2^{-1}g_3) \le l(g_2g_3^{-1})+l(g_1)$ by \Cref{overlap}. It follows that $L(X^{1}_{\{\},\{3\}})<L(X)$, so $X$ is not minimal. 

If $\Delta > \min\{l(g_2),l(g_3)\}$, then either $g_2^{-1}g_3$ fixes $p$ (if $l(g_2)=l(g_3)$) or it is hyperbolic and its axis contains $p$ (otherwise); see \Cref{bigint}. Thus $X$ is not minimal in this case either, as \cite[Proposition 1.7]{P} (respectively \Cref{overlap}) shows that $l(g_1^{-1}g_2^{-1}g_3) \le l(g_2g_3^{-1})+l(g_1)$.

Hence we may suppose that $\Delta=\min\{l(g_2),l(g_3)\}$, and $l(g_2)\neq l(g_3)$ (as otherwise $X$ is not minimal by the above argument). Let us first assume that $\Delta=l(g_2)<l(g_3)$. Consider the product replacements $g_1 \mapsto g_1g_2^{-1}$ and $g_3 \mapsto g_3g_2^{-1}$, which change the terms of $L(X)$ as in (\ref{1to12and3to32}). As before, $l(g_2^2g_3^{-1})\le l(g_2g_3)-l(g_2)$ and, similarly, $l(g_1g_2^{-2})\le l(g_1g_2)-l(g_2)$. It remains to compare $l(g_2^{-1}g_3g_2^{-1}g_1)$ with $l(g_1g_3)$.

By \Cref{overlap}, $l(g_2^{-1}g_3)=\max\{0,l(g_3)-l(g_2)-2\Delta'\},$ where $\Delta'$ is the length of the path $\gamma_3 \cap g_2 \cdot \gamma_3$. Observe that if $g_2^{-1}g_3$ is hyperbolic, then its axis is at distance $\Delta'$ from the initial vertex $p$ of $\gamma_2\cap \gamma_3$; see the lower diagram of \Cref{case4} for the case where $\Delta'>\delta_3$. If $g_2^{-1}g_3$ is elliptic, then by \Cref{ellfix} it fixes a vertex of $\gamma_3$ at distance $\frac{l(g_3)-l(g_2)}{2}$ from $p$. In either case, $d(p,g_2^{-1}g_3p) \le l(g_3)-l(g_2)$ and hence
\begin{align*}
l(g_2^{-1}g_3g_2^{-1}g_1)&\le d(p,g_2^{-1}g_3g_2^{-1}g_1p) \\
&\le d(p,g_2^{-1}g_3p)+d(p,g_2^{-1}g_1p)\\
&\le l(g_1g_3)-l(g_1)-l(g_2)+d(p,g_2^{-1}g_1p).
\end{align*}

If $l(g_1)<l(g_2)$, then $d(p,g_2^{-1}g_1p)=l(g_2)-l(g_1)$ as the axis of $g_2^{-1}g_1$ contains $p$; see the lower diagram of \Cref{case4}. If $l(g_1)=l(g_2)$, then $g_2^{-1}g_1$ fixes $p$ and $d(p,g_2^{-1}g_1p)=0$. If $l(g_1)>l(g_2)$, then \Cref{overlap} shows that $l(g_2^{-1}g_1)=\max\{0,l(g_1)-l(g_2)-2\Delta''\}$, where $\Delta''$ is the length of the path $\gamma_1 \cap g_2 \cdot \gamma_1$. Either $g_2^{-1}g_1$ is hyperbolic with axis at distance $\Delta''$ from $p$ (see \Cref{bigint3}), or $g_2^{-1}g_1$ fixes a vertex at distance $\frac{l(g_1)-l(g_2)}{2}$ from $p$ by \Cref{ellfix}. Hence $d(p,g_2^{-1}g_1p)\le l(g_1)-l(g_2)$. In each of these three cases, this shows that $L(X^2_{\{\},\{1,3\}})<L(X)$ and so $X$ is not minimal if $\Delta= l(g_2)<l(g_3)$. 

Now suppose that $\Delta=l(g_3)<l(g_2)$. Consider the product replacements $g_1 \mapsto g_1g_3^{-1}$ and $g_2 \mapsto g_2g_3^{-1}$, which preserve all terms in the expression of $L(X)$ except for the following:
\begin{align}\label{1to13and2to23}
l(g_1g_3) &\mapsto l(g_1) \mapsto l(g_1g_3^{-1}) \mapsto l(g_1g_3^{-2})  \notag \\
l(g_2g_3) &\mapsto l(g_2) \mapsto l(g_2g_3^{-1}) \mapsto l(g_2g_3^{-2})   \\
l(g_1g_2)  &\mapsto l(g_1g_3^{-1}g_2g_3^{-1})=l(g_3^{-1}g_2g_3^{-1}g_1). \notag 
\end{align}
By \Cref{overlap}, $l(g_1g_3^{-2})\le l(g_1g_3)-l(g_3)$ and $l(g_2g_3^{-2})\le l(g_2g_3)-l(g_3)$, so it remains to compare $l(g_3^{-1}g_2g_3^{-1}g_1)$ with $l(g_1g_2)$. \Cref{overlap} shows that $l(g_3^{-1}g_2)=\max\{0,l(g_2)-l(g_3)-2\Delta'\}$, where $\Delta'$ is the length of the path $\gamma_2 \cap g_3 \cdot \gamma_2$. A similar argument to the previous case shows that 
\begin{align*}
l(g_3^{-1}g_2g_3^{-1}g_1)&\le d(p,g_3^{-1}g_2p)+d(p,g_3^{-1}g_1p)\\
&\le l(g_1g_2)-l(g_1)-l(g_3)+d(p,g_3^{-1}g_1p).
\end{align*}

If $l(g_1)<l(g_3)$, then the axis of $g_3^{-1}g_1$ contains $p$ (see \Cref{bigint}) and $d(p,g_3^{-1}g_1p)=l(g_3)-l(g_1)$ by \Cref{overlap}. On the other hand, if $l(g_1)=l(g_3)$, then $g_3^{-1}g_1$ fixes $p$ and so $d(p,g_3^{-1}g_1p)=0$. Hence $L(X^3_{\{\},\{1,2\}})<L(X)$, and overall we conclude that $X$ is not minimal when $\Delta(\gamma_1,\gamma_3) \ge l(g_1)$.

Finally, we suppose that $\Delta(\gamma_1,\gamma_3) < l(g_1)$ and $\Delta \ge l(g_2)$. Consider the product replacements $g_1 \mapsto g_1g_2^{-1}$ and $g_3 \mapsto g_3g_2^{-1}$, which change the terms of $L(X)$ as in (\ref{1to12and3to32}). \Cref{overlap} shows that $l(g_1g_2^{-2})\le l(g_1g_2)-l(g_2)$ and $l(g_2^2g_3^{-1})\le l(g_2g_3)-l(g_2)$, so it remains to compare $l(g_2^{-1}g_3g_2^{-1}g_1)$ and $l(g_1g_3)$. By \Cref{overlap}, $l(g_2^{-1}g_3)=\max\{0,l(g_3)-l(g_2)-2\Delta'\}$, where $\Delta'$ is the length of the path $\gamma_3 \cap g_2 \cdot \gamma_3$ if $\Delta=l(g_2)$, and $\Delta'=0$ otherwise. As in the previous case, the initial vertex $p$ of $\gamma_2 \cap \gamma_3=\gamma_1 \cap \gamma_2$ is such that
\begin{align*}
l(g_2^{-1}g_3g_2^{-1}g_1)&\le d(p,g_2^{-1}g_3p)+d(p,g_2^{-1}g_1p)\\
&\le l(g_1g_3)-l(g_1)-l(g_2)+d(p,g_2^{-1}g_1p).
\end{align*}

Similarly, $l(g_2^{-1}g_1)=\max\{0,l(g_1)-l(g_2)-2\Delta''\}$, where $\Delta''$ is the length of the path $\gamma_1 \cap g_2 \cdot \gamma_1$ if $\Delta=l(g_2)$, and $\Delta''=0$ otherwise. Since $g_2^{-1}g_1$ is either hyperbolic with axis at distance $\Delta''$ from $p$ (see \Cref{bigint3}), or elliptic and fixes a vertex at distance $\frac{l(g_1)-l(g_2)}{2}$ from $p$ (by \Cref{ellfix}), it follows that $d(p,g_2^{-1}g_1p)\le l(g_1)-l(g_2)$. Thus $L(X^2_{\{\},\{1,3\}})<L(X)$ and $X$ is not minimal.
\end{proof}

\begin{lem}\label{3threepairs2}
Let $X=(g_1,g_2,g_3)$ be a triple of hyperbolic isometries of a simplicial tree such that $\gamma_1 \cap \gamma_2, \gamma_1 \cap \gamma_3$ and $\gamma_2 \cap \gamma_3$ are non-empty, distinct and all contained in a single axis. If $X$ does not satisfy the hypotheses of the Ping Pong Lemma, then it is not minimal.
\end{lem}
\begin{proof}
Without loss of generality, suppose that $\gamma_3$ contains all three paths of intersection. Note that $\Delta(\gamma_1,\gamma_3), \Delta(\gamma_2,\gamma_3)>\Delta(\gamma_1,\gamma_2)$, as otherwise we are in the case described by \Cref{3threepairs1}. We may further suppose that $\gamma_3$ intersects both $\gamma_1$ and $\gamma_2$ with the same orientation, and that $g_3$ translates $\gamma_1$ towards $\gamma_2$. Let $\Delta_3$ denote the length of the path ${\rm Proj}_{\gamma_3}(\gamma_1) \cup {\rm Proj}_{\gamma_3}(\gamma_2)$, so that $\Delta_3=\Delta(\gamma_1,\gamma_3)+\Delta(\gamma_2,\gamma_3)-\Delta(\gamma_1,\gamma_2)$. 

Let us first suppose that $\Delta(\gamma_1,\gamma_3)<l(g_1)$ and $\Delta(\gamma_2,\gamma_3)<l(g_2)$. Since $X$ does not satisfy the hypotheses of the Ping Pong Lemma, $\Delta_3 \ge l(g_3)$. We may also assume (after swapping the roles of $g_1$ and $g_2$, if necessary) that $\Delta(\gamma_1,\gamma_3) \ge \Delta(\gamma_2,\gamma_3)$, and hence $l(g_3) \le 2\Delta(\gamma_1,\gamma_3)-\Delta(\gamma_1,\gamma_2)$. Consider the product replacement $g_1 \mapsto g_1g_3^{-1}$, which preserves all terms in the expression of $L(X)$ except for the following:
\begin{align}\label{1to13}
l(g_1g_3) &\mapsto l(g_1) \mapsto l(g_1g_3^{-1}) \mapsto l(g_1g_3^{-2})  \notag \\
l(g_1g_2)  &\mapsto l(g_1^{-1}g_2^{-1}g_3)  \\
l(g_1g_2^{-1})  &\mapsto l(g_1^{-1}g_2g_3). \notag 
\end{align}
\Cref{overlap} shows that 
\begin{align*}
l(g_1g_3^{-2})&\le l(g_1g_3)+l(g_3)-2\Delta(\gamma_1,\gamma_3) \\
& \le l(g_1g_3)-\Delta(\gamma_1,\gamma_2).
\end{align*} 
The axis of $g_1^{-1}g_2^{-1}$ (respectively $g_1^{-1}g_2$) intersects $\gamma_3$ with opposite orientations along a subpath of length $\Delta_3$ (respectively $\Delta(\gamma_1,\gamma_3)-\Delta(\gamma_1,\gamma_2)$); see the upper diagram of \Cref{case5}. It follows from \Cref{overlap} that $l(g_1^{-1}g_2^{-1}g_3) \le l(g_1g_2)-l(g_3)$ and 
\begin{align*}
l(g_1^{-1}g_2g_3) &\le l(g_1g_2^{-1})+l(g_3)+2\Delta(\gamma_1,\gamma_2)-2\Delta(\gamma_1,\gamma_3) \\
&\le l(g_1g_2^{-1})+\Delta(\gamma_1,\gamma_2).
\end{align*}
Hence $L(X^3_{\{\},\{1\}})<L(X)$, so $X$ is not minimal if $\Delta(\gamma_1,\gamma_3)<l(g_1)$ and $\Delta(\gamma_2,\gamma_3)<l(g_2)$.

After switching the roles of $g_1$ and $g_2$, if necessary, we may assume for the remainder of the proof that $\Delta(\gamma_1,\gamma_3) \ge l(g_1)$. Suppose first that $l(g_1) \le l(g_3)$. If $\Delta(\gamma_2,\gamma_3) \neq \min\{l(g_2),l(g_3)\}$, then consider the product replacement $g_3 \mapsto g_3g_1^{-1}$, which changes the terms of $L(X)$ as in (\ref{3to31}). As before, $l(g_1^2g_3^{-1})\le l(g_1g_3)-l(g_1)$. Moreover, the axis of $g_2g_3$  intersects the axis of $g_1^{-1}$ with opposite orientations along a subpath of length either $\Delta(\gamma_1,\gamma_3)$ (if $l(g_3)>\Delta(\gamma_1,\gamma_3)-\Delta(\gamma_1,\gamma_2)$) or $l(g_3)+\Delta(\gamma_1,\gamma_2)$ (otherwise); see the lower diagram of \Cref{case5} for the former case. In either case, \Cref{overlap} shows that $l(g_1^{-1}g_2g_3) \le l(g_2g_3)-l(g_1)$.

If $\Delta(\gamma_2,\gamma_3)< \min\{l(g_2),l(g_3)\}$, then the axis of $g_2^{-1}g_3$ intersects $\gamma_1$; see the lower diagram of \Cref{case5} for the case where $\Delta_3>l(g_3)$. On the other hand, if $\Delta(\gamma_2,\gamma_3)> \min\{l(g_2),l(g_3)\}$, then either $l(g_2)=l(g_3)$ and $g_2^{-1}g_3$ fixes the initial vertex $p$ of $\gamma_2 \cap \gamma_3$, or $l(g_2)\neq l(g_3)$ the axis of $g_2^{-1}g_3$ contains $p$; see \Cref{bigint}. Since $p \in \gamma_1$, \cite[Proposition 1.7]{P} (respectively \Cref{overlap}) shows that $l(g_1^{-1}g_2^{-1}g_3) \le l(g_2g_3^{-1})+l(g_1)$. Thus $L(X^1_{\{\},\{3\}})<L(X)$, and so $X$ is not minimal if $\Delta(\gamma_2,\gamma_3) \neq \min\{l(g_2),l(g_3)\}$.

We may now suppose that $\Delta(\gamma_2,\gamma_3)= \min\{l(g_2),l(g_3)\}$. If $l(g_2)=l(g_3)$, then $g_2^{-1}g_3$ fixes $p$ and $X$ is not minimal as above. So let us first assume that $\Delta(\gamma_2,\gamma_3)=l(g_2)<l(g_3)$. Consider the two product replacements $g_1 \mapsto g_1g_2^{-1}$ and $g_3 \mapsto g_3g_2^{-1}$, which change the terms of $L(X)$ as in (\ref{1to12and3to32}). \Cref{overlap} shows that $l(g_1g_2^{-2})\le l(g_1g_2)+l(g_2)-2\Delta(\gamma_1,\gamma_2)$ and $l(g_2^2g_3^{-1})\le l(g_2g_3)-l(g_2)$, so it remains to compare $l(g_2^{-1}g_3g_2^{-1}g_1)$ and $l(g_1g_3)$.

By \Cref{overlap}, $l(g_2^{-1}g_3)=\max\{0,l(g_3)-l(g_2)-2\Delta'\}$, where $\Delta'$ is the length of the path $\gamma_3 \cap g_2 \cdot \gamma_3$. Either the axis of $g_2^{-1}g_3$ is at distance $\Delta'$ from $p$ (see \Cref{bigint3}), or $g_2^{-1}g_3$ fixes a vertex at distance $\frac{l(g_3)-l(g_2)}{2}$ from $p$ by \Cref{ellfix}. As before, this shows that $d(p,g_2^{-1}g_3p) \le l(g_3)-l(g_2)$ and hence
\begin{align*}
l(g_2^{-1}g_3g_2^{-1}g_1)&\le d(p,g_2^{-1}g_3p)+d(p,g_2^{-1}g_1p)\\
&\le l(g_1g_3)-l(g_1)-l(g_2)+d(p,g_2^{-1}g_1p).
\end{align*}

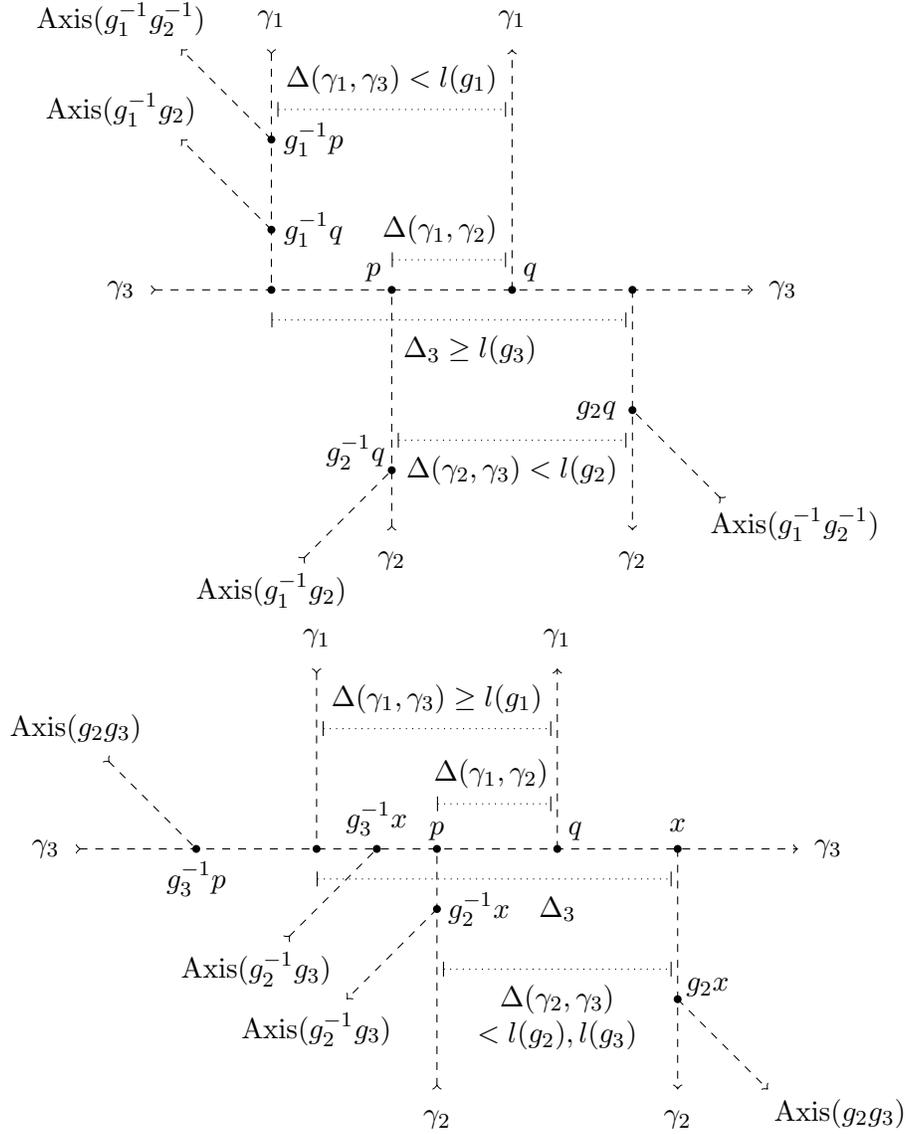
\begin{figure}[h]
\centering

\begin{tikzpicture}
  [scale=0.8,auto=left] 

\draw[dashed,>->] (0,0) to (10,0); \node at (-0.5,0) {$\gamma_3$}; \node at (10.5,0) {$\gamma_3$};
\draw[dashed,>-] (4,-4) to (4,0); \node at (4,-4.5) {$\gamma_2$}; \draw[dashed,>-] (2,4) to (2,0); \node at (2,4.5) {$\gamma_1$};
\draw[dashed,->] (8,0) to (8,-4); \node at (8,-4.5) {$\gamma_2$}; \draw[dashed,->] (6,0) to (6,4); \node at (6,4.5) {$\gamma_1$};
 \node[circle,inner sep=0pt,minimum size=3,fill=black] (1) at (2, 0) {}; 
  \node[circle,inner sep=0pt,minimum size=3,fill=black] (1) at (4, 0) {}; \node at (3.7,0.3) {$p$};
\node[circle,inner sep=0pt,minimum size=3,fill=black] (1) at (6, 0) {}; \node at (6.3,0.3) {$q$};
\node[circle,inner sep=0pt,minimum size=3,fill=black] (1) at (8, 0) {};
\draw[dotted,|-|] (4,0.5) to (5.9,0.5); \node at (4.8,1) {$\Delta(\gamma_1,\gamma_2)$};
\draw[dotted,|-|] (2.1,3) to (5.9,3); \node at (4,3.5) {$\Delta(\gamma_1,\gamma_3)<l(g_1)$};
\draw[dotted,|-|] (4.1,-2.5) to (7.9,-2.5); \node at (6,-3) {$\Delta(\gamma_2,\gamma_3)<l(g_2)$};
\draw[dotted,|-|] (2,-0.5) to (7.9,-0.5); \node at (5.3,-1) {$\Delta_3 \ge l(g_3)$};

  \node[circle,inner sep=0pt,minimum size=3,fill=black] (1) at (2, 2.5) {}; \node at (2.7,2.5) {$g_1^{-1}p$};
 \node[circle,inner sep=0pt,minimum size=3,fill=black] (1) at (8,-2) {}; \node at (7.4,-2) {$g_2q$};
 \draw[dashed, >-] (9.5,-3.5) to (8,-2); \node at (10.7,-3.9) {${\rm Axis}(g_1^{-1}g_2^{-1})$};
  \draw[dashed, ->] (2,2.5) to (0.5,4); \node at (-0.5,4.5) {${\rm Axis}(g_1^{-1}g_2^{-1})$};
  
    \node[circle,inner sep=0pt,minimum size=3,fill=black] (1) at (2, 1) {}; \node at (2.7,1) {$g_1^{-1}q$};
      \draw[dashed, ->] (2,1) to (0.5,2.5); \node at (-0.5,3) {${\rm Axis}(g_1^{-1}g_2)$};
      
\node[circle,inner sep=0pt,minimum size=3,fill=black] (1) at (4, -3) {}; \node at (3.4,-2.7) {$g_2^{-1}q$};
      \draw[dashed, >-] (2.5,-4.5) to (4,-3); \node at (2,-5) {${\rm Axis}(g_1^{-1}g_2)$};

\end{tikzpicture}

\begin{tikzpicture}
  [scale=0.8,auto=left] 

\draw[dashed,>->] (-2,0) to (10,0); \node at (-2.5,0) {$\gamma_3$}; \node at (10.5,0) {$\gamma_3$};
\draw[dashed,>-] (4,-4) to (4,0); \node at (4,-4.5) {$\gamma_2$}; \draw[dashed,>-] (2,3) to (2,0); \node at (2,3.5) {$\gamma_1$};
\draw[dashed,->] (8,0) to (8,-4); \node at (8,-4.5) {$\gamma_2$}; \draw[dashed,->] (6,0) to (6,3); \node at (6,3.5) {$\gamma_1$};
 \node[circle,inner sep=0pt,minimum size=3,fill=black] (1) at (2, 0) {}; 
  \node[circle,inner sep=0pt,minimum size=3,fill=black] (1) at (4, 0) {}; \node at (4,0.3) {$p$};
\node[circle,inner sep=0pt,minimum size=3,fill=black] (1) at (6, 0) {}; \node at (6.3,0.3) {$q$};
\node[circle,inner sep=0pt,minimum size=3,fill=black] (1) at (8, 0) {};\node at (8,0.4) {$x$};
\draw[dotted,|-|] (4,0.75) to (5.9,0.75); \node at (4.9,1.25) {$\Delta(\gamma_1,\gamma_2)$};
\draw[dotted,|-|] (2.1,2) to (5.9,2); \node at (4,2.5) {$\Delta(\gamma_1,\gamma_3)\ge l(g_1)$};
\draw[dotted,|-|] (4.1,-2) to (7.9,-2); \node at (6,-2.5) {$\Delta(\gamma_2,\gamma_3)$}; \node at (6,-3.1) {$<l(g_2),l(g_3)$};
\draw[dotted,|-|] (2,-0.5) to (7.9,-0.5); \node at (6,-1) {$\Delta_3$};

  \node[circle,inner sep=0pt,minimum size=3,fill=black] (1) at (0, 0) {}; \node at (0,-0.5) {$g_3^{-1}p$};
 \node[circle,inner sep=0pt,minimum size=3,fill=black] (1) at (8,-2.5) {}; \node at (8.5,-2.3) {$g_2x$};
 \draw[dashed, ->] (8,-2.5) to (9.5,-4); \node at (10.7,-4.4) {${\rm Axis}(g_2g_3)$};
  \draw[dashed, >-] (-1.5,1.5) to (0,0); \node at (-2,2) {${\rm Axis}(g_2g_3)$};
  
    \node[circle,inner sep=0pt,minimum size=3,fill=black] (1) at (3, 0) {}; \node at (3,0.5) {$g_3^{-1}x$};
      \draw[dashed, >-] (1.5,-1.5) to (3,0); \node at (1,-2) {${\rm Axis}(g_2^{-1}g_3)$};
      
\node[circle,inner sep=0pt,minimum size=3,fill=black] (1) at (4, -1) {}; \node at (4.7,-1) {$g_2^{-1}x$};
      \draw[dashed, ->]  (4,-1) to (2.5,-2.5); \node at (2,-3) {${\rm Axis}(g_2^{-1}g_3)$};

\end{tikzpicture}

\caption{All three pairs of axes $\gamma_1, \gamma_2, \gamma_3$ intersect along distinct subpaths, and one axis contains all three paths of intersection.} 
\label{case5}

\end{figure}

If $\Delta(\gamma_1,\gamma_2)<l(g_1)$, then $g_2^{-1}g_1$ is hyperbolic and its axis contains $p$ (see \Cref{smallint}) so $d(p,g_2^{-1}g_1p)=l(g_1)+l(g_2)-2\Delta(\gamma_1,\gamma_2)$ by \Cref{overlap}. If $\Delta(\gamma_1,\gamma_2)\ge l(g_1)$, then $l(g_2^{-1}g_1)=\max\{0,l(g_2)-l(g_1)-2\Delta''\}$ by \Cref{overlap}, where $\Delta''$ is the length of the path $\gamma_2 \cap g_1 \cdot \gamma_2$ if $\Delta(\gamma_1,\gamma_2)= l(g_1)$, and $\Delta''=0$ otherwise. \Cref{bigint}, \Cref{bigint3} and \Cref{ellfix} show that either the axis of $g_2^{-1}g_1$ is at distance $\Delta''$ from $p$, or $g_2^{-1}g_1$ fixes a vertex at distance $\frac{l(g_2)-l(g_1)}{2}$ from $p$. Thus $d(p,g_2^{-1}g_1p)\le l(g_2)-l(g_1)$. So $L(X^2_{\{\},\{1,3\}})<L(X)$, and $X$ is not minimal when $\Delta(\gamma_2,\gamma_3)=l(g_2)<l(g_3)$.

Hence we may suppose that $\Delta(\gamma_2,\gamma_3)=l(g_3)<l(g_2)$. Now consider the product replacements $g_1 \mapsto g_1g_3^{-1}$ and $g_2 \mapsto g_2g_3^{-1}$, which change the terms of $L(X)$ as in (\ref{1to13and2to23}). Since $l(g_1g_3^{-2})\le l(g_1g_3)+l(g_3)-2\Delta(\gamma_1,\gamma_3)$ and $l(g_2g_3^{-2}) \le l(g_2g_3)-l(g_3)$ by \Cref{overlap}, it remains to compare $l(g_3^{-1}g_2g_3^{-1}g_1)$ and $l(g_1g_2)$. Note that $l(g_3^{-1}g_2)=\max\{0,l(g_2)-l(g_3)-2\Delta'\}$ by \Cref{overlap}, where $\Delta'$ is the length of the path $\gamma_2 \cap g_3 \cdot \gamma_2$. After swapping the roles of $g_2$ and $g_3$ in the above case, it is straightforward to verify that 
\begin{align*}
l(g_3^{-1}g_2g_3^{-1}g_1)&\le d(p,g_3^{-1}g_2p)+d(p,g_3^{-1}g_1p)\\
&\le l(g_1g_2)-l(g_1)-l(g_3)+d(p,g_3^{-1}g_1p).
\end{align*}
If $l(g_1)=l(g_3)$, then $l(g_1)>\Delta(\gamma_1,\gamma_2)$ and $g_3^{-1}g_1$ fixes a vertex at distance $l(g_1)-\Delta(\gamma_1,\gamma_2)$ from $p$. Hence $d(p,g_3^{-1}g_1p)\le l(g_1)+l(g_3)-2\Delta(\gamma_1,\gamma_2)$. If $l(g_1)<l(g_3)$, then $l(g_3^{-1}g_1)=\max\{0,l(g_3)-l(g_1)-2\Delta''\}$, where $\Delta''$ is the length of the path $\gamma_3 \cap g_1 \cdot \gamma_3$ if $\Delta(\gamma_1,\gamma_3)=l(g_1)$, and $\Delta''=0$ otherwise. If $g_3^{-1}g_1$ is hyperbolic, then its axis either contains $p$ (if $l(g_1) \le \Delta(\gamma_1,\gamma_2)$) or is at distance $\Delta''+l(g_1)-\Delta(\gamma_1,\gamma_2)$ from $p$ (otherwise); see \Cref{bigint,bigint3}. If $g_3^{-1}g_1$ is elliptic, then it either fixes $p$ (if $l(g_1) \le \Delta(\gamma_1,\gamma_2)$) or fixes a vertex at distance $\frac{l(g_3)-l(g_1)}{2}+l(g_1)-\Delta(\gamma_1,\gamma_2)$ from $p$ by \Cref{ellfix}. In either case, $d(p,g_3^{-1}g_1p)< l(g_1)+l(g_3)$ and hence $L(X^3_{\{\},\{1,2\}})<L(X)$. We conclude that $X$ is not minimal when $l(g_1) \le l(g_3)$.

Finally, we suppose that $l(g_3)<l(g_1)$. Consider the product replacements $g_1 \mapsto g_1g_3^{-1}$ and $g_2 \mapsto g_2g_3^{-1}$, which change the terms of $L(X)$ as in (\ref{1to13and2to23}). Since $l(g_1g_3^{-2})< l(g_1g_3)-l(g_3)$ and $l(g_2g_3^{-2})\le l(g_2g_3)+l(g_3)-2\Delta(\gamma_2,\gamma_3)$ by \Cref{overlap}, it remains to compare $l(g_1g_3^{-1}g_2g_3^{-1})$ and $l(g_1g_2)$. Note that $l(g_1g_3^{-1})=l(g_1)-l(g_3)$, and the axis of $g_1g_3^{-1}$ contains the terminal vertex $q$ of $\gamma_1 \cap \gamma_2$; see \Cref{bigint}. Hence
\begin{align*}
l(g_1g_3^{-1}g_2g_3^{-1})&\le d(q,g_1g_3^{-1}q)+d(q,g_2g_3^{-1}q)\\
&\le l(g_1g_2)-l(g_2)-l(g_3)+d(q,g_2g_3^{-1}q).
\end{align*}

If $\Delta(\gamma_2,\gamma_3)<\min\{l(g_2),l(g_3)\}$, then the axis of $g_2g_3^{-1}$ is at distance $\Delta(\gamma_2,\gamma_3)-\Delta(\gamma_1,\gamma_2)$ from $q$ (see \Cref{smallint}) and thus \Cref{overlap} shows that $d(q,g_2g_3^{-1}q)=l(g_2)+l(g_3)-2\Delta(\gamma_1,\gamma_2)$. Hence $L(X^3_{\{\},\{1,2\}})<L(X)$, so $X$ is not minimal. We may therefore assume for the remainder of the proof that $\Delta(\gamma_2,\gamma_3) \ge \min\{l(g_2),l(g_3)\}$.

If $l(g_2)=l(g_3)$, then $g_2g_3^{-1}$ either fixes $q$ (if $l(g_2) \le \Delta(\gamma_1,\gamma_2)$) or fixes a vertex at distance $l(g_2)-\Delta(\gamma_1,\gamma_2)$ from $q$ (otherwise). In either case, $d(q,g_2g_3^{-1}q) \le l(g_2)+l(g_3)-2\Delta(\gamma_1,\gamma_2)$ and so $X$ is not minimal. So assume that $l(g_2)<l(g_3)$. By \Cref{overlap}, $l(g_2g_3^{-1})=\max\{0,l(g_3)-l(g_2)-2\Delta'\}$, where $\Delta'$ is the length of $\gamma_3 \cap g_2 \cdot \gamma_3$ if $\Delta(\gamma_2,\gamma_3)=l(g_2)$, and $\Delta'=0$ otherwise. If $g_2g_3^{-1}$ is hyperbolic, then its axis either contains $q$ (when $l(g_2) \le \Delta(\gamma_1,\gamma_2)$) or is at distance $\Delta'+l(g_2)-\Delta(\gamma_1,\gamma_2)$ from $q$ (otherwise); see \Cref{bigint,bigint3}. Hence $d(q,g_2g_3^{-1}q) < l(g_2)+l(g_3)$. On the other hand, if $g_2g_3^{-1}$ is elliptic, then by \Cref{ellfix} it fixes a vertex at distance $\frac{l(g_3)-l(g_2)}{2}+l(g_2)-\Delta(\gamma_1,\gamma_2)$ from $q$, and $d(q,g_2g_3^{-1}q)< l(g_2)+l(g_3)$. In both cases, $X$ is not minimal, so we may finally suppose that $l(g_3)<l(g_2)$.

By \Cref{overlap}, $l(g_2g_3^{-1})=\max\{0,l(g_2)-l(g_3)-2\Delta'\}$, where $\Delta'$ is the length of $\gamma_2 \cap g_3 \cdot \gamma_2$ if $\Delta(\gamma_2,\gamma_3)=l(g_3)$, and $\Delta'=0$ otherwise. By swapping the roles of $g_2$ and $g_3$ in the previous case, we again see that $d(q,g_2g_3^{-1}q)< l(g_2)+l(g_3)$ and this completes the proof.
\end{proof}

\begin{lem}\label{3threepairs3}
Let $X=(g_1,g_2,g_3)$ be a triple of hyperbolic isometries of a simplicial tree such that $\gamma_1 \cap \gamma_2, \gamma_1 \cap \gamma_3$ and $\gamma_2 \cap \gamma_3$ are non-empty and distinct, but not all contained in a single axis. If $X$ does not satisfy the hypotheses of the Ping Pong Lemma, then it is not minimal.
\end{lem}
\begin{proof}
We may assume that $\gamma_1$ and $\gamma_2$ both intersect $\gamma_3$ with the same orientation, and that $g_3$ translates towards $\gamma_1$ towards $\gamma_2$. Firstly we suppose that $\Delta(\gamma_i,\gamma_j)<\min\{l(g_i),l(g_j)\}$ for all $i,j$. Let $\Delta_1$ denote the length of the shortest subpath of $\gamma_1$ containing both ${\rm Proj}_{\gamma_1}(\gamma_2)$ and ${\rm Proj}_{\gamma_1}(\gamma_3)$, so that $\Delta_1=\Delta(\gamma_1,\gamma_2)+\Delta(\gamma_1,\gamma_3)$. Similarly, we define $\Delta_2=\Delta(\gamma_1,\gamma_2)+\Delta(\gamma_2,\gamma_3)$ and $\Delta_3=\Delta(\gamma_1,\gamma_3)+\Delta(\gamma_2,\gamma_3)$.

Since $X$ does not satisfy the hypotheses of the Ping Pong Lemma, we may assume (by symmetry) that $l(g_3) \le \Delta_3$. After swapping $g_1$ and $g_2$, if necessary, we may also assume that $\Delta(\gamma_1,\gamma_3)\ge \Delta(\gamma_2,\gamma_3)$ and hence $l(g_3) \le 2\Delta(\gamma_1,\gamma_3)$. If $l(g_3)<\Delta_3-\Delta(\gamma_1,\gamma_2)$, then consider the replacement $g_1 \mapsto g_3g_1g_3^{-1}$, which changes the terms of $L(X)$ as in (\ref{1to313}). Since $g_3 \cdot \gamma_1$ intersects $\gamma_2$ along a path of length $\Delta_3-l(g_3)>\Delta(\gamma_1,\gamma_2)$, it follows from \Cref{overlap} that $L(X^3_{\{1\},\{1\}})<L(X)$ and $X$ is not minimal. So we may assume that $l(g_3) \ge \Delta_3-\Delta(\gamma_1,\gamma_2)$. 

Suppose now that $\Delta_1 \ge l(g_1)$ and $\Delta_2 \ge l(g_2)$. If $l(g_1)<\Delta_1-\Delta(\gamma_2,\gamma_3)$, then we may argue as above that $L(X^1_{\{3\},\{3\}})<L(X)$ and $X$ is not minimal. Thus we may assume that $l(g_1) \ge \Delta_1-\Delta(\gamma_2,\gamma_3)$. Similarly, we may assume that $l(g_2) \ge \Delta_2-\Delta(\gamma_1,\gamma_3)$. But then
\begin{align*}
l(g_1)+l(g_2)+l(g_3) & \ge \Delta(\gamma_1,\gamma_2)+\Delta(\gamma_1,\gamma_3)+\Delta(\gamma_2,\gamma_3) \\
& > l(g_1)+l(g_2)+l(g_3),
\end{align*} 
which is a contradiction. So either $\Delta_1<l(g_1)$ or $\Delta_2<l(g_2)$.

Consider the product replacement $g_1 \mapsto g_1g_3^{-1}$, which changes the terms of $L(X)$ as in (\ref{1to13}). As before, $l(g_1g_3^{-2}) \le l(g_1g_3)+l(g_3)-2\Delta(\gamma_1,\gamma_3)$. Observe that the axis of $g_1^{-1}g_2$ intersects $\gamma_3$ with opposite orientations along a path of length $\Delta(\gamma_1,\gamma_3)$; see the upper diagram of \Cref{case6}. Thus \Cref{overlap} shows that $l(g_1^{-1}g_2g_3)\le l(g_1g_2^{-1})+l(g_3)-2\Delta(\gamma_1,\gamma_3).$

\begin{figure}[h]
\centering

\begin{tikzpicture}
  [scale=0.8,auto=left] 
\draw[dashed,>->] (-2,0) to (11,0); \node at (-2.5,0) {$\gamma_3$}; \node at (11.5,0) {$\gamma_3$};
\draw[dashed,>-]  (7,4) to (5,2) to (5,0); \node at (7.3,4.3) {$\gamma_2$}; \draw[dashed,>-] (-1,3) to (2,0); \node at (-1.3,3.3) {$\gamma_1$};
\draw[dashed,->] (8,0) to (11,3); \node at (11.3,3.3) {$\gamma_2$}; \draw[dashed,->] (5,2) to (3,4) ; \node at (2.7,4.3) {$\gamma_1$};
 \node[circle,inner sep=0pt,minimum size=3,fill=black] (1) at (2, 0) {}; \node at (2,-0.4) {$p$};
\node[circle,inner sep=0pt,minimum size=3,fill=black] (1) at (5, 0) {}; \node at (5,-0.4) {$q$};
  \node[circle,inner sep=0pt,minimum size=3,fill=black] (1) at (5, 2) {}; \node at (5,2.4) {$x$};
\node[circle,inner sep=0pt,minimum size=3,fill=black] (1) at (8, 0) {};

\draw[dotted,|-|] (2,-1.5) to (8,-1.5); \node at (5,-2) {$\Delta_3 \ge l(g_3)$};
\draw[dotted,|-|] (2,-3) to (4.9,-3); \node at (2.8,-3.5) {$l(g_1)> \Delta(\gamma_1,\gamma_3)$};
\draw[dotted,|-|] (5.1,-3) to (8,-3); \node at (6.5,-3.5) {$\Delta(\gamma_2,\gamma_3)$};

\draw[dotted,|-|] (2.1,0.3) to (4.7,0.3) to (4.7,1.9); \node at (3.4,0.9) {$l(g_1)>\Delta_1$};
\draw[dotted,|-|] (7.9,0.3) to (5.3,0.3) to (5.3,1.9); \node at (6.6,0.9) {$\Delta_2 > l(g_2)$};

 \node[circle,inner sep=0pt,minimum size=3,fill=black] (1) at (0.5, 1.5) {}; \node at (1.2,1.7) {$g_1^{-1}q$};
  \draw[dashed, ->] (0.5,1.5) to (0.5,3.5); \node at (0.5,4) {${\rm Axis}(g_1^{-1}g_2)$};
 \node[circle,inner sep=0pt,minimum size=3,fill=black] (1) at (6, 3) {}; \node at (6.7,2.8) {$g_2^{-1}x$};
  \draw[dashed, >-] (6,5) to (6,3); \node at (6,5.5) {${\rm Axis}(g_1^{-1}g_2)$};
  
   \node[circle,inner sep=0pt,minimum size=3,fill=black] (1) at (7, 0) {}; \node at (6.6,-0.4) {$g_2x$};
     \draw[dashed, >-] (8.5,-1.5) to (7,0); \node at (10,-1.7) {${\rm Axis}(g_1^{-1}g_2^{-1})$};
 \node[circle,inner sep=0pt,minimum size=3,fill=black] (1) at (0, 0) {}; \node at (0,0.4) {$g_1x$};
      \draw[dashed, ->] (0,0) to (-1.5,-1.5); \node at (-2,-2) {${\rm Axis}(g_1^{-1}g_2^{-1})$};

\end{tikzpicture}

\begin{tikzpicture}
  [scale=0.8,auto=left] 

\draw[dashed,>->] (-2,0) to (11,0); \node at (-2.5,0) {$\gamma_3$}; \node at (11.5,0) {$\gamma_3$};
\draw[dashed,>-]  (7,4) to (5,2) to (5,0); \node at (7.3,4.3) {$\gamma_2$}; \draw[dashed,>-] (-1,3) to (2,0); \node at (-1.3,3.3) {$\gamma_1$};
\draw[dashed,->] (8,0) to (11,3); \node at (11.3,3.3) {$\gamma_2$}; \draw[dashed,->] (5,2) to (3,4) ; \node at (2.7,4.3) {$\gamma_1$};
 \node[circle,inner sep=0pt,minimum size=3,fill=black] (1) at (2, 0) {}; \node at (2,-0.4) {$p$};
\node[circle,inner sep=0pt,minimum size=3,fill=black] (1) at (5, 0) {}; \node at (5,-0.4) {$q$};
  \node[circle,inner sep=0pt,minimum size=3,fill=black] (1) at (5, 2) {}; \node at (5,2.4) {$x$};
\node[circle,inner sep=0pt,minimum size=3,fill=black] (1) at (8, 0) {};

\draw[dotted,|-|] (2,-3) to (4.9,-3); \node at (2.8,-3.5) {$l(g_1) < \Delta(\gamma_1,\gamma_3)$};
\draw[dotted,|-|] (5.1,-3) to (8,-3); \node at (6.5,-3.5) {$\Delta(\gamma_2,\gamma_3)$};
\draw[dotted,|-|] (5.3,0.1) to (5.3,2); \node at (6.5,1) {$\Delta(\gamma_1,\gamma_2)$};

 \node[circle,inner sep=0pt,minimum size=3,fill=black] (1) at (3.5, 0) {}; \node at (4,0.5) {$g_1^{-1}q$};
  \draw[dashed, ->] (3.5,0) to (2,1.5); \node at (2,2) {${\rm Axis}(g_1^{-1}g_2)$};
 \node[circle,inner sep=0pt,minimum size=3,fill=black] (1) at (6, 3) {}; \node at (6.7,2.8) {$g_2^{-1}x$};
  \draw[dashed, >-] (6,5) to (6,3); \node at (6,5.5) {${\rm Axis}(g_1^{-1}g_2)$};
  
    \draw[dashed, ->] (3.5,0) to (2,-1.5); \node at (2,-2) {${\rm Axis}(g_1^{-1}g_3)$};
     \node[circle,inner sep=0pt,minimum size=3,fill=black] (1) at (0, 0) {}; \node at (0,0.5) {$g_3^{-1}g_1p$};
       \draw[dashed, >-] (-1.5,-1.5) to (0,0); \node at (-2,-2) {${\rm Axis}(g_1^{-1}g_3)$};

\end{tikzpicture}

\caption{All three pairs of axes $\gamma_1, \gamma_2,\gamma_3$ intersect, and no axis contains all three paths of intersection.}
\label{case6}
\end{figure}
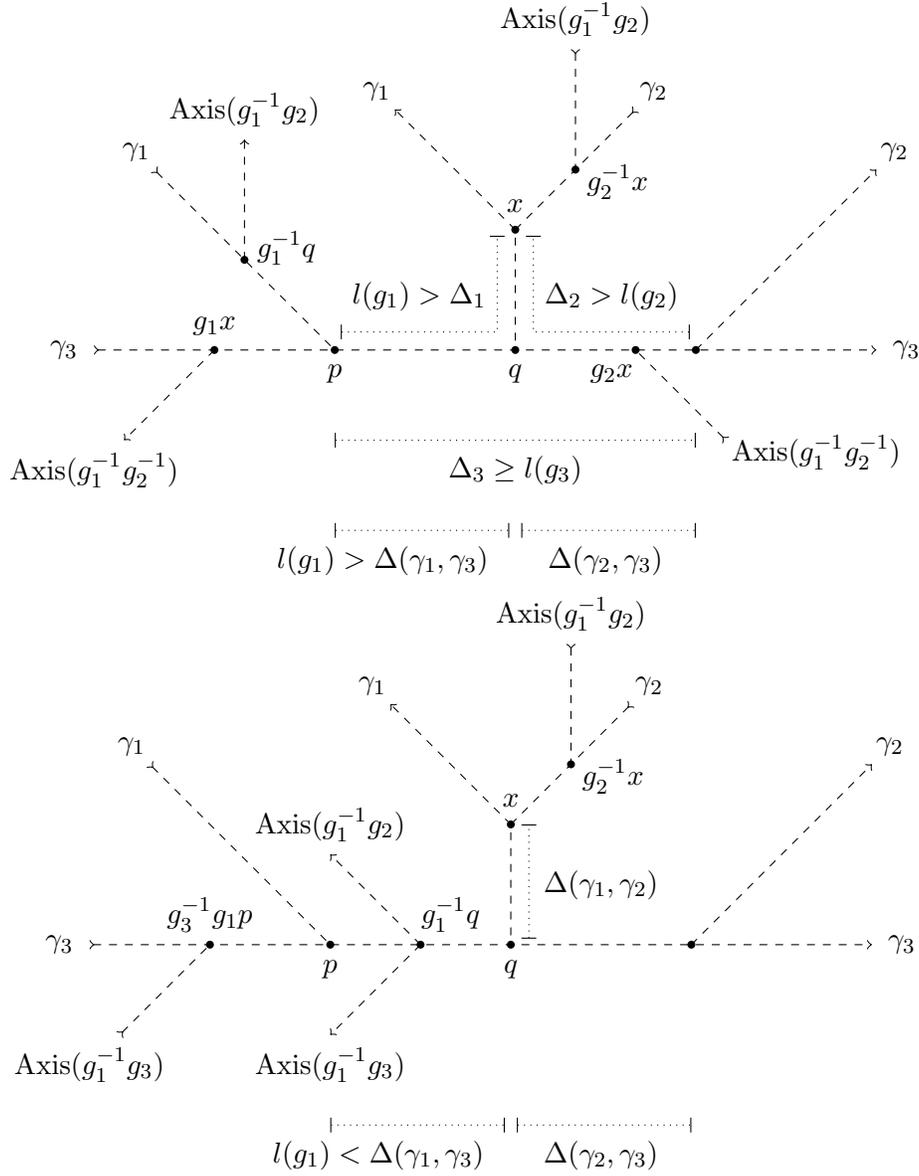

If $\Delta_1<l(g_1)$, then the axis of $g_1^{-1}g_2^{-1}$ intersects $\gamma_3$ with opposite orientations along a subpath of length at least $\Delta(\gamma_1,\gamma_3)$; see the upper diagram of \Cref{case6} for the case where $\Delta_2 > l(g_2)$. \Cref{overlap} shows that 
\begin{align*}
l(g_1^{-1}g_2^{-1}g_3) &\le l(g_1g_2)+l(g_3)-2\Delta(\gamma_1,\gamma_3) \\
&\le l(g_1g_2).
\end{align*}

Similarly, if $\Delta_2<l(g_2)$, then the axis of $g_1^{-1}g_2^{-1}$ intersects $\gamma_3$ with opposite orientations along a subpath of length at least $\Delta(\gamma_2,\gamma_3)$ and \Cref{overlap} shows that $l(g_1^{-1}g_2^{-1}g_3) \le l(g_1g_2)+l(g_3)-2\Delta(\gamma_2,\gamma_3)$. In either case, it follows that $L(X^3_{\{\},\{1\}})<L(X)$ since $l(g_3)\le\Delta_3$. We conclude that $X$ is not minimal when $\Delta(\gamma_i,\gamma_j)<\min\{l(g_i),l(g_j)\}$ for all $i,j$.

Without loss of generality, we may now suppose that $\Delta(\gamma_1,\gamma_3)\ge l(g_1)$. If $l(g_1) \le l(g_3)$, then consider the product replacements $g_2 \mapsto g_2g_1^{-1}$ and $g_3 \mapsto g_3g_1^{-1}$, which preserve all terms in the expression of $L(X)$ except for the following:
\begin{align}\label{2to21and3to31}
l(g_1g_2) &\mapsto l(g_2) \mapsto l(g_1g_2^{-1}) \mapsto l(g_1^2g_2^{-1})  \notag \\
l(g_1g_3) &\mapsto l(g_3) \mapsto l(g_1g_3^{-1}) \mapsto l(g_1^2g_3^{-1})   \\
l(g_2g_3)  &\mapsto l(g_2g_1^{-1}g_3g_1^{-1})=l(g_1^{-1}g_3g_1^{-1}g_2). \notag 
\end{align}
\Cref{overlap} shows that $l(g_1^2g_2^{-1}) \le l(g_1g_2)+l(g_1)-2\Delta(\gamma_1,\gamma_2)$ and $l(g_1^2g_3^{-1}) \le l(g_1g_3)-l(g_1)$, so it remains to compare $l(g_1^{-1}g_3g_1^{-1}g_2)$ and $l(g_2g_3)$. By \Cref{overlap}, $l(g_1^{-1}g_3)=\max\{0,l(g_3)-l(g_1)-2\Delta'\}$, where $\Delta'$ is the length of the path $\gamma_3 \cap g_1 \cdot \gamma_3$ if $\Delta(\gamma_1,\gamma_3)=l(g_1)$, and $\Delta'=0$ otherwise. Let $q$ denote the terminal vertex of the path $\gamma_1 \cap \gamma_3$. 

If $g_1^{-1}g_3$ is hyperbolic, then its axis is at distance $\Delta'$ from $g_1^{-1}q$; see the lower diagram of \Cref{case6} for the case $\Delta(\gamma_1,\gamma_3)>l(g_1)$ (and hence $\Delta'=0$), and see \Cref{bigint3} otherwise. Since $g_1^{-1}g_2$ is hyperbolic and its axis contains $g_1^{-1}q$, \Cref{overlap} shows that 
\begin{align*}
l(g_1^{-1}g_3g_1^{-1}g_2)&\le l(g_1^{-1}g_3)+l(g_1^{-1}g_2)+2\Delta' \\
&= l(g_2g_3).
\end{align*}

On the other hand, if $g_1^{-1}g_3$ is elliptic, then by \Cref{ellfix} it either fixes $g_1^{-1}q$ (if $l(g_1)=l(g_3)$) or a vertex at distance $\frac{l(g_3)-l(g_1)}{2}$ from $g_1^{-1}q$ (otherwise). Hence \cite[Proposition 1.7]{P} shows that
\begin{align*}
l(g_1^{-1}g_3g_1^{-1}g_2)&\le l(g_1^{-1}g_3)+l(g_1^{-1}g_2)+l(g_3)-l(g_1) \\
&= l(g_2g_3).
\end{align*}

In either case, $X$ is not minimal, so we may suppose that $l(g_3)<l(g_1)$. Consider the product replacements $g_1 \mapsto g_1g_3^{-1}$ and $g_2 \mapsto g_2g_3^{-1}$, which changes the terms of $L(X)$ as in (\ref{1to13and2to23}). It follows from \Cref{overlap} that $l(g_1g_3^{-2})< l(g_1g_3)-l(g_3)$ and $l(g_2g_3^{-2})\le l(g_2g_3)+l(g_3)-2\Delta(\gamma_2,\gamma_3)$, so it remains to compare $l(g_1g_3^{-1}g_2g_3^{-1})$ and $l(g_1g_2)$. By \Cref{overlap} $l(g_1g_3^{-1})=l(g_1)-l(g_3)$, and note that the axis of $g_1g_3^{-1}$ contains the terminal vertex $q$ of $\gamma_1 \cap \gamma_3$; see \Cref{bigint}. If $\Delta(\gamma_2,\gamma_3)<\min\{l(g_2),l(g_3)\}$, then $g_2g_3^{-1}$ is hyperbolic with axis at distance $\Delta(\gamma_2,\gamma_3)$ from $q$ (which is also the initial vertex of $\gamma_2 \cap \gamma_3$); see \Cref{smallint}. Thus \Cref{overlap} shows
\begin{align*}
l(g_1g_3^{-1}g_2g_3^{-1}) &\le l(g_1g_3^{-1})+l(g_2g_3^{-1})+2\Delta(\gamma_2,\gamma_3) \\
&=l(g_1g_2)
\end{align*}
and it follows that $L(X^3_{\{\},\{1,2\}})<L(X)$, so $X$ is not minimal. Hence we may assume for the remainder of the proof that $\Delta(\gamma_2,\gamma_3) \ge \min\{l(g_2),l(g_3)\}$.

Let us first suppose that $l(g_2) \le l(g_3)$. \Cref{overlap} shows that $l(g_2g_3^{-1})=\max\{0,l(g_3)-l(g_2)-2\Delta'\}$, where $\Delta'$ is the length of the path $\gamma_3 \cap g_2 \cdot \gamma_3$ if $\Delta(\gamma_2,\gamma_3)=l(g_2)$, and $\Delta'=0$ otherwise. If $g_2g_3^{-1}$ is hyperbolic, then its axis is at distance $l(g_2)+\Delta'$ from $q$; see \Cref{bigint,bigint3}. Thus \Cref{overlap} shows that
\begin{align*}
l(g_1g_3^{-1}g_2g_3^{-1}) &= l(g_1g_3^{-1})+l(g_2g_3^{-1})+2(l(g_2)+\Delta') \\
&=l(g_1g_2).
\end{align*}
On the other hand, if $g_2g_3^{-1}$ is elliptic, then by \Cref{ellfix} it fixes a vertex at distance $l(g_2)+\frac{l(g_3)-l(g_2)}{2}$ from $q$ and \cite[Proposition 1.7]{P} shows that 
\begin{align*}
l(g_1g_3^{-1}g_2g_3^{-1}) &\le l(g_1g_3^{-1})+l(g_2g_3^{-1})+2l(g_2)+l(g_3)-l(g_2) \\
&=l(g_1g_2).
\end{align*}
We conclude that $X$ is not minimal if $l(g_2) \le l(g_3)$.

Similarly, if $l(g_3)<l(g_2)$, then $l(g_2g_3^{-1})=\max\{0,l(g_2)-l(g_3)-2\Delta'\}$, where $\Delta'$ is the length of the path $\gamma_2 \cap g_3 \cdot \gamma_2$ if $\Delta(\gamma_2,\gamma_3)=l(g_3)$, and $\Delta'=0$ otherwise. Either the axis of $g_2g_3^{-1}$ is at distance $l(g_3)+\Delta'$ from $q$, or $g_2g_3^{-1}$ fixes a vertex at distance $l(g_3)+\frac{l(g_2)-l(g_3)}{2}$ from $q$, and \Cref{overlap} (respectively \cite[Proposition 1.7]{P}) again shows that $l(g_1g_3^{-1}g_2g_3^{-1}) \le l(g_1g_2)$. Hence $X$ is not minimal if $\Delta(\gamma_2,\gamma_3) \ge \min\{l(g_2),l(g_3)\}$.
\end{proof}

We conclude by proving \Cref{thm:conj}:

\begin{proof}[Proof of Theorem $\ref{thm:conj}$]
Let $X$ be a pair or triple of hyperbolic isometries of $T$ which do not satisfy the hypotheses of the Ping Pong Lemma. Since $X$ must satisfy the hypotheses of exactly one of Lemmas \ref{2}-\ref{3threepairs3}, we conclude that $X$ is not minimal.
\end{proof}

\section*{Acknowledgements}
The author would like to thank Eamonn O'Brien, Jeroen Schillewaert and the referee for their feedback on previous versions of this paper. The author is supported by a University of Auckland FRDF grant and the Woolf Fisher Trust.

\bibliographystyle{plain}

\bibliography{references}

\newpage
\appendix
\section{The axes of two hyperbolic isometries}
\label{app}

\renewcommand{\thesection}{\Alph{section}}
\setcounter{figure}{0}

Let $g_1$ and $g_2$ be hyperbolic isometries of a simplicial tree, with axes denoted by $\gamma_1$ and $\gamma_2$. Suppose that $l(g_1)\le l(g_2)$, and $\gamma_1 \cap \gamma_2$ is either empty or finite. The following six figures indicate the possible interactions between $\gamma_1, \gamma_2$. In the first five, $g_1g_2$ is hyperbolic and the axis of $g_1g_2$ is shown. The last demonstrates a case where $g_1g_2$ is elliptic and fixes a vertex.

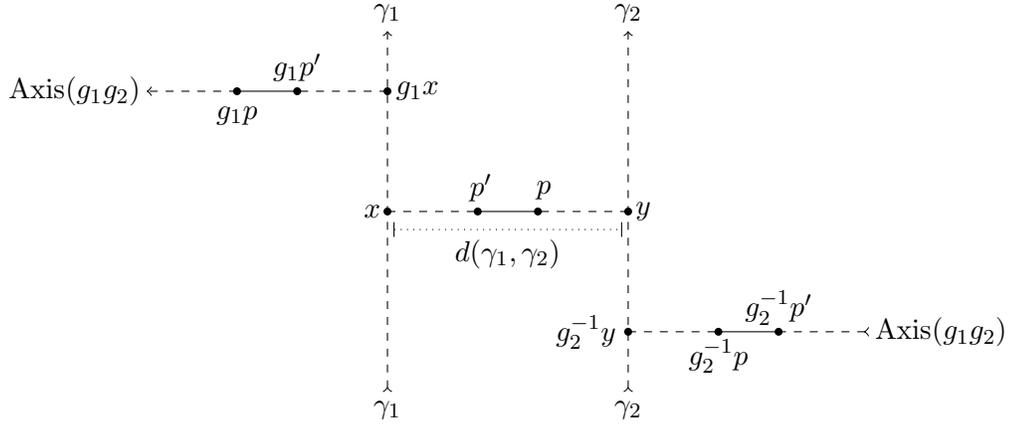
\begin{figure}[h]
\centering
\begin{tikzpicture}
  [scale=0.8,auto=left] 

    \node[circle,inner sep=0pt,minimum size=3,fill=black] (1) at (-4,0) {};
   \node[circle,inner sep=0pt,minimum size=3,fill=black] (1) at (-4,2) {};
                       \draw [dashed, >->] (-4,-3) to (-4,3);
       \node at (-4, 3.3) {$\gamma_1$}; \node at (-4, -3.3) {$\gamma_1$};
          \node at (-4.25,0) {$x$}; \node at (-3.5,2) {$g_1x$};  

     \draw [dotted, |-|] (-3.9,-0.3) to (-0.1,-0.3); \node at (-2, -0.7) {$d(\gamma_1,\gamma_2)$};
     
         \node[circle,inner sep=0pt,minimum size=3,fill=black] (1) at (0,0) {};
              \draw [dashed,>->] (0,-3) to (0,3) ;
                            \node[circle,inner sep=0pt,minimum size=3,fill=black] (1) at (0,-2) {};
\node at (0, 3.3) {$\gamma_2$};  \node at (0, -3.3) {$\gamma_2$};
  \node at (0.25,0) {$y$};  \node at (-0.7, -2) {$g_2^{-1}y$};

 \node[circle,inner sep=0pt,minimum size=3,fill=black] (1) at (-2.5,0) {};
  \node[circle,inner sep=0pt,minimum size=3,fill=black] (1) at (-1.5,0) {};
  \draw [dashed] (-4,0) to (-2.5,0); \draw (-2.5,0) to (-1.5,0); \draw [dashed] (0,0) to (-1.5,0);
  \node at (-2.45,0.4) {$p'$};   \node at (-1.4,0.35) {$p$};
 
  \node[circle,inner sep=0pt,minimum size=3,fill=black] (1) at (-5.5,2) {};
   \node[circle,inner sep=0pt,minimum size=3,fill=black] (1) at (-6.5,2) {};
  \draw (-5.5, 2) to (-6.5, 2); \draw [dashed] (-4, 2) to (-5.5,2);
  \node at (-5.5,2.4) {$g_1p'$}; \node at (-6.5,1.6) {$g_1p$};
   \draw [dashed,->] (-6.5,2) to (-8,2); \node at (-9.2,2) {${\rm Axis}(g_1g_2)$};
   
    \node[circle,inner sep=0pt,minimum size=3,fill=black] (1) at (1.5,-2) {};
   \node[circle,inner sep=0pt,minimum size=3,fill=black] (1) at (2.5,-2) {};
  \draw (1.5, -2) to (2.5, -2); \draw [dashed] (0, -2) to (1.5,-2);
  \node at (1.5,-2.4) {$g_2^{-1}p$}; \node at (2.5,-1.6) {$g_2^{-1}p'$};
    \draw [dashed,>-] (4,-2) to (2.5,-2); \node at (5.2,-2) {${\rm Axis}(g_1g_2)$};

\end{tikzpicture} 
\caption{$\gamma_1$ and $\gamma_2$ do not intersect.} \label{noint}
\end{figure}

\begin{figure}[h]
\centering
\begin{tikzpicture}
  [scale=0.8,auto=left]

 \node[circle,inner sep=0pt,minimum size=3,fill=black] (1) at (11,-1) {};
  \node[circle,inner sep=0pt,minimum size=3,fill=black] (1) at (11,1) {};
  \node[circle,inner sep=0pt,minimum size=3,fill=black] (1) at (10,1.5) {};
  \node[circle,inner sep=0pt,minimum size=3,fill=black] (1) at (12,1.5) {};
  \node[circle,inner sep=0pt,minimum size=3,fill=black] (1) at (12,-1.5) {};
   \node[circle,inner sep=0pt,minimum size=3,fill=black] (1) at (10,-1.5) {};
     \draw [dashed] (11,-1) to (11,1); \draw (12,1.5) to (11,1) to (10,1.5);  \draw (12,-1.5) to (11,-1) to (10,-1.5);

\draw [dashed,->]  (12,1.5) to (12.6,4.5);  \draw [dashed,->]  (10,1.5) to (9.4,4.5);
\draw [dashed,>-] (9.4, -4.5) to (10,-1.5);   \draw [dashed,>-]  (12.6,-4.5) to  (12,-1.5);
 \node at (9.4, 4.8) {$\gamma_1$}; \node at (12.6, 4.8) {$\gamma_2$};  \node at (9.4, -4.8) {$\gamma_1$}; \node at (12.6, -4.8) {$\gamma_2$};
 
  \node at (9.7, -1.5) {$p'$};  \node at (12.4, 1.5) {$q'$};  \node at (11, -1.4) {$p$};   \node at (11, 1.4) {$q$};
  
  \node[circle,inner sep=0pt,minimum size=3,fill=black] (1) at (12.4,-3.5) {}; \node at (11.7,-3.5) {$g_2^{-1}p$};
   \node[circle,inner sep=0pt,minimum size=3,fill=black] (1) at (13.4,-3.5) {};   \draw (12.4,-3.5) to (13.4,-3.5);
   \node at (13.4,-3.9) {$g_2^{-1}p'$};   \draw [dashed,>-] (14.4,-3.5) to (13.4,-3.5);
\node at (15.6, -3.5) {${\rm Axis}(g_1g_2)$};
            
    \node[circle,inner sep=0pt,minimum size=3,fill=black] (1) at (9.8,2.5) {}; \node at (10.35,2.5) {$g_1q$};
      \node[circle,inner sep=0pt,minimum size=3,fill=black] (1) at (8.8,2.5) {};   \draw (8.8,2.5) to (9.8,2.5);  
          \node at (8.8,2.9) {$g_1q'$}; 
              \draw [dashed,->]  (8.8,2.5) to (7.8,2.5);  
            \node at (6.6, 2.5) {${\rm Axis}(g_1g_2)$};

\end{tikzpicture} 
\caption{$\gamma_1$ and $\gamma_2$ intersect with the same orientation.} \label{overlapsame}
\end{figure}
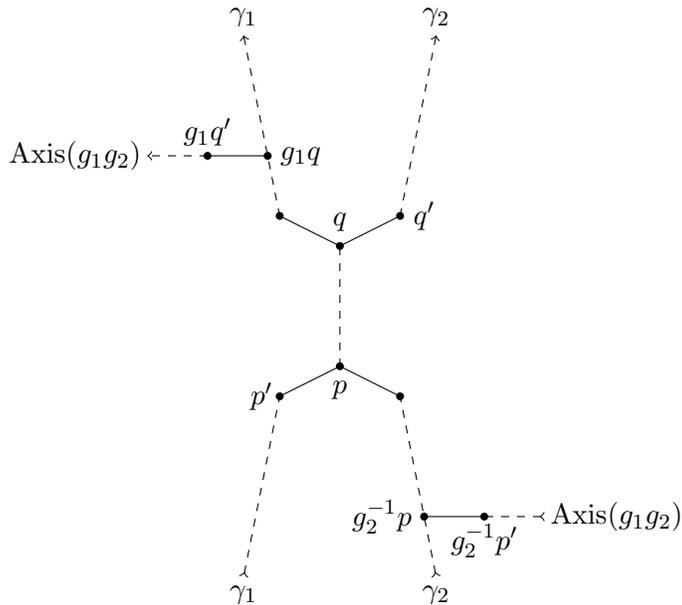

\begin{figure}[h]
\centering
\begin{tikzpicture}
  [scale=0.8,auto=left] 
 \node[circle,inner sep=0pt,minimum size=3,fill=black] (1) at (11,-1) {};
  \node[circle,inner sep=0pt,minimum size=3,fill=black] (1) at (11,1) {};
  \node[circle,inner sep=0pt,minimum size=3,fill=black] (1) at (10,1.5) {};
  \node[circle,inner sep=0pt,minimum size=3,fill=black] (1) at (12,1.5) {};
  \node[circle,inner sep=0pt,minimum size=3,fill=black] (1) at (12,-1.5) {};
   \node[circle,inner sep=0pt,minimum size=3,fill=black] (1) at (10,-1.5) {};
     \draw [dashed] (11,-1) to (11,1); \draw (12,1.5) to (11,1) to (10,1.5);  \draw (12,-1.5) to (11,-1) to (10,-1.5); 
     
          \draw [dotted, |-|] (11.3,-1) to (11.3,1); \node at (12.4, 0) {$\Delta < l(g_1)$};
                  
\draw [dashed,>-]  (12.8,5.5) to (12,1.5);  \draw [dashed,->]  (10,1.5) to (9.2,5.5);
\draw [dashed,>-] (9.4, -4.5) to (10,-1.5);   \draw [dashed,->]  (12,-1.5) to (12.6,-4.5);
 \node at (9.2, 5.8) {$\gamma_1$}; \node at (12.8, 5.8) {$\gamma_2$};  \node at (9.4, -4.8) {$\gamma_1$}; \node at (12.6, -4.8) {$\gamma_2$};
 
  \node at (9.7, -1.5) {$p'$};  \node at (12.4, -1.5) {$p''$};   \node at (11, -1.4) {$p$};

  \node[circle,inner sep=0pt,minimum size=3,fill=black] (1) at (12.6,4.5) {}; \node at (11.8,4.5) {$g_2^{-1}p$};
    \node[circle,inner sep=0pt,minimum size=3,fill=black] (1) at (9.6,3.5) {}; \node at (10.05,3.5) {$g_1p$};
      \node[circle,inner sep=0pt,minimum size=3,fill=black] (1) at (12.4,3.5) {}; \node at (11.6,3.5) {$g_2^{-1}p''$};
    \node[circle,inner sep=0pt,minimum size=3,fill=black] (1) at (9.8,2.5) {}; \node at (10.3,2.5) {$g_1p'$};
    \draw (9.8,2.5) to (9.6,3.5);   \draw (12.6,4.5) to (12.4,3.5);
   \node[circle,inner sep=0pt,minimum size=3,fill=black] (1) at (13.6,4.5) {};   \draw (12.6,4.5) to (13.6,4.5);
      \node[circle,inner sep=0pt,minimum size=3,fill=black] (1) at (8.6,3.5) {};   \draw (8.6,3.5) to (9.6,3.5);  
          \node at (8.6,3.9) {$g_1p''$};  \node at (13.6,4.9) {$g_2^{-1}p'$};
              \draw [dashed,->] (8.6,3.5) to (7.6,3.5);     \draw [dashed,>-] (14.6,4.5) to (13.6,4.5);
            \node at (6.4, 3.5) {${\rm Axis}(g_1g_2)$}; \node at (15.8, 4.5) {${\rm Axis}(g_1g_2)$};

\end{tikzpicture} 
\caption{$\gamma_1$ and $\gamma_2$ intersect with opposite orientations along a path of length $\Delta<\min\{l(g_1),l(g_2)\}=l(g_1).$} \label{smallint}
\end{figure}

\begin{figure}[h]
\centering
\begin{tikzpicture}
  [scale=0.8,auto=left] 

\draw [dotted, |-|] (12.4,2) to (12.4,-1); \node at (13.7,0.5) {$\Delta > l(g_1)$};

 \node[circle,inner sep=0pt,minimum size=3,fill=black] (1) at (11,-1) {};
  \node[circle,inner sep=0pt,minimum size=3,fill=black] (1) at (11,2) {};
  \node[circle,inner sep=0pt,minimum size=3,fill=black] (1) at (10,2.5) {};
  \node[circle,inner sep=0pt,minimum size=3,fill=black] (1) at (12,2.5) {};
  \node[circle,inner sep=0pt,minimum size=3,fill=black] (1) at (12,-1.5) {};
   \node[circle,inner sep=0pt,minimum size=3,fill=black] (1) at (10,-1.5) {};
     \draw [dashed] (11,-1) to (11,2); \draw (12,2.5) to (11,2) to (10,2.5);  \draw (12,-1.5) to (11,-1) to (10,-1.5); 
                  
\draw [dashed,>-]  (12.4,4.5) to (12,2.5);  \draw [dashed,->]  (10,2.5) to (9.6,4.5);
\draw [dashed,>-] (9.6, -3.5) to (10,-1.5);   \draw [dashed,->]  (12,-1.5) to  (12.4,-3.5);
 \node at (9.6, 4.8) {$\gamma_1$}; \node at (12.4, 4.8) {$\gamma_2$};  \node at (9.6, -3.8) {$\gamma_1$}; \node at (12.4, -3.8) {$\gamma_2$};
 
 \node at (12.4, -1.5) {$p'$};  \node at (12.4, 2.5) {$q'$};  \node at (11, -1.4) {$p$};   \node at (11, 2.4) {$q$};

    \node[circle,inner sep=0pt,minimum size=3,fill=black] (1) at (11,0) {}; \node at (11.5,0) {$g_1p$};
      \node[circle,inner sep=0pt,minimum size=3,fill=black] (1) at (10,0) {};   \draw (10,0) to (11,0);  
          \node at (10,-0.4) {$g_1p'$}; 
              \draw [dashed,->] (10,0) to (8.5,0);  
            \node at (7.3, 0) {${\rm Axis}(g_1g_2)$};
            
  \node[circle,inner sep=0pt,minimum size=3,fill=black] (1) at (12.2,3.5) {}; \node at (11.15,3.55) {$g_2^{-1}g_1^{-1}q$};
   \node[circle,inner sep=0pt,minimum size=3,fill=black] (1) at (13.2,3.5) {};   \draw (12.2,3.5) to (13.2,3.5);
          \node at (13.45,3.9) {$g_2^{-1}g_1^{-1}q'$};
               \draw [dashed,>-] (14.7,3.5) to (13.2,3.5);
        \node at (15.9, 3.5) {${\rm Axis}(g_1g_2)$};

\end{tikzpicture} 
\caption{$\gamma_1$ and $\gamma_2$ intersect with opposite orientations along a path of length $\Delta>\min\{l(g_1),l(g_2)\}=l(g_1)\neq l(g_2).$} \label{bigint}
\end{figure}

\begin{figure}[h]
\centering
\begin{tikzpicture}
  [scale=0.8,auto=left]

 \node[circle,inner sep=0pt,minimum size=3,fill=black] (1) at (3,-1) {};
 \node[circle,inner sep=0pt,minimum size=3,fill=black] (1) at (3,1) {};
 \node[circle,inner sep=0pt,minimum size=3,fill=black] (1) at (3.6,2.5) {};
 \node[circle,inner sep=0pt,minimum size=3,fill=black] (1) at (3.6,-2.5) {};
     \draw [dashed] (3,-1) to (3,1);  
     \draw [dashed] (3,1) to (3.6,2.5) ;   \draw [dashed] (3,-1) to (3.6,-2.5) ;         
  \node at (0.7, 6.8) {$\gamma_1$}; \node at (1.1, -5.8) {$\gamma_1$}; 
\draw [dashed,->]  (3,1) to (0.8,6.5);         \draw [dashed,>-] (1.2, -5.5) to (3,-1);

   \node[circle,inner sep=0pt,minimum size=3,fill=black] (1) at (4,3.5) {};
 \node[circle,inner sep=0pt,minimum size=3,fill=black] (1) at (4.4,4.5) {};
  \node[circle,inner sep=0pt,minimum size=3,fill=black] (1) at (4.8,5.5) {};

   \node[circle,inner sep=0pt,minimum size=3,fill=black] (1) at (4,-3.5) {};
      \draw (3.6,-2.5) to (4,-3.5);    
   \node at (5.3, 6.8) {$\gamma_2$}; \node at (4.9, -5.8) {$\gamma_2$};                    
   \draw [dashed,>-]  (5.2,6.5) to (4.8,5.5); \draw [dashed,->] (4, -3.5) to (4.8, -5.5);          
        
 \draw (3.6,2.5) to (4,3.5) ;      \draw (4.4,4.5) to (4.8,5.5) ;     \draw [dashed] (4,3.5) to (4.4,4.5);

  \node[circle,inner sep=0pt,minimum size=3,fill=black] (1) at (4.9,-2.5) {};      \draw (3.6,-2.5) to (4.9, -2.5);
    \node[circle,inner sep=0pt,minimum size=3,fill=black] (1) at (4.9,2.5) {};         \draw (3.6,2.5) to (4.9, 2.5);                            
    \node[circle,inner sep=0pt,minimum size=3,fill=black] (1) at (6.1,5.5) {};         \draw (4.8,5.5) to (6.1, 5.5);   

  \node at (2.7, -1) {$p$}; \node at (2.5, 1) {$g_1p$}; 
   \node at (3.3, -2.5) {$r$};  \node at (3.6, -3.4) {$r''$};  \node at (5, -2.1) {$r'$};
      \node at (3.1, 2.5) {$g_1r$};  \node at (3.3, 3.6) {$g_1r'$};  \node at (5, 2.1) {$g_1r''$};     
 \node at (3.6, 4.6) {$g_2^{-1}r''$};     \node at (4.2, 5.6) {$g_2^{-1}r$}; \node at (6.1,5.9) {$g_2^{-1}r'$};  
                                  
 \draw [dashed, >-] (7.6,5.5) to (6.1,5.5); \node at (8.8,5.5) {${\rm Axis}(g_1g_2)$};                 
  \draw [dashed, ->] (4.9,2.5) to (6.4,2.5); \node at (7.6,2.5) {${\rm Axis}(g_1g_2)$};              
          
\node[circle,inner sep=0pt,minimum size=3,fill=black] (1) at (2.2,-3) {};   \node at (1.6, -2.6) {$g_1^{-1}p$};
 \draw [dashed, ->] (2.2,-3) to (-0.3,-3); \node at (-1.3,-3) {$g_1^{-1}\cdot \gamma_2$};                 
  \draw [dashed, >-] (6.4,-2.5) to (4.9,-2.5); \node at (7.4,-2.5) {$g_1^{-1}\cdot \gamma_2$};                                       
       
  \node at (0.6, 0) {$l(g_1)=\Delta$}; \draw [dotted, |-|] (1.7, 1) to (1.7, -1);
  \node at (5.2,-1.3) {$\Delta'< \frac{l(g_2)-l(g_1)}{2}$}; \draw [dotted, |-|] (3.25,-0.9) to (3.85,-2.35);

\end{tikzpicture}
\caption{$\gamma_1$ and $\gamma_2$ intersect with opposite orientations along a path of length $\Delta=\min\{l(g_1), l(g_2)\}=l(g_1)$. The axes $\gamma_2$ and $g_1^{-1}\cdot \gamma_2$ (and hence $\gamma_2$ and $g_1 \cdot \gamma_2$) intersect along a path of length $\Delta' <\frac{l(g_2)-l(g_1)}{2}.$} \label{bigint3}
\end{figure}
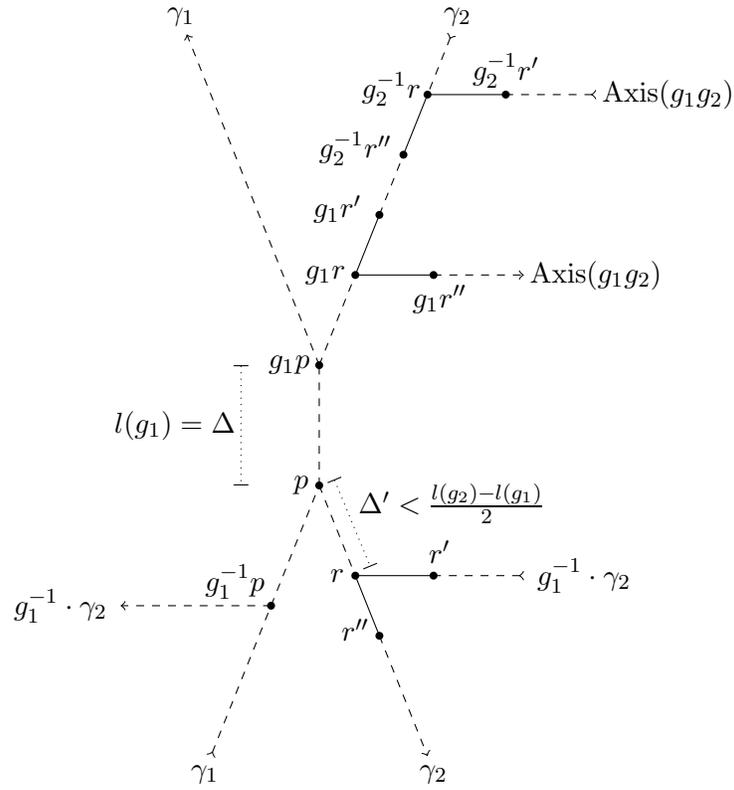

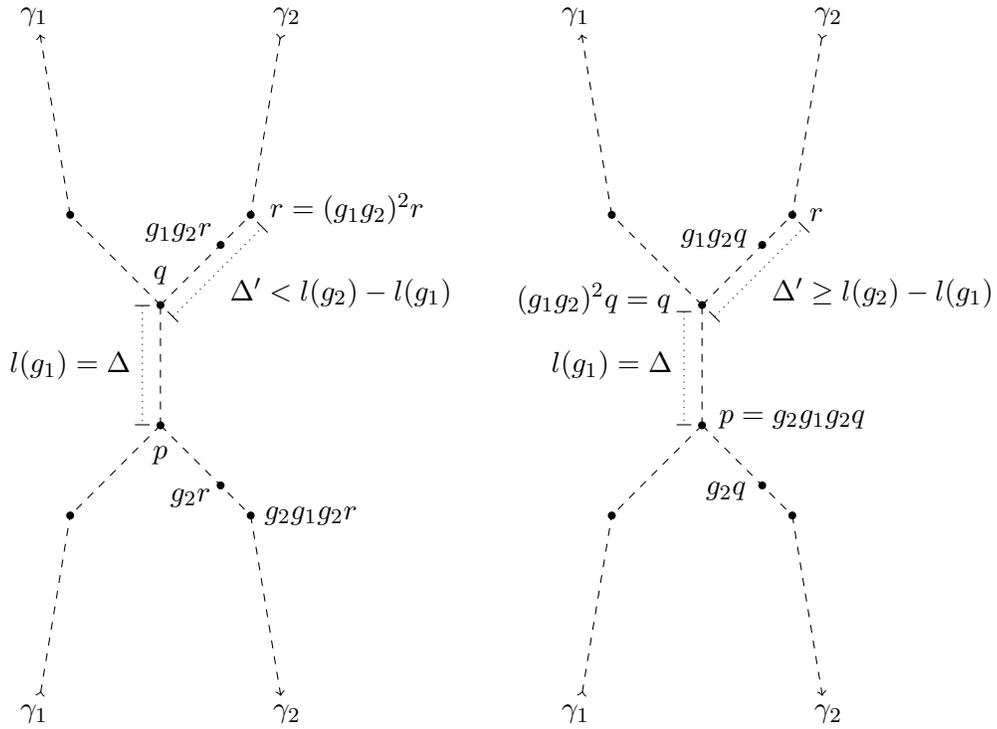
\begin{figure}[h]
\centering
\begin{tikzpicture}
  [scale=0.8,auto=left] 

 \node[circle,inner sep=0pt,minimum size=3,fill=black] (1) at (-3,-1) {};
 \node[circle,inner sep=0pt,minimum size=3,fill=black] (1) at (-3,1) {};
\draw [dashed] (-3,-1) to (-3,1); \node at (-3, 1.5) {$q$}; \node at (-3, -1.5) {$p$}; 
 \node[circle,inner sep=0pt,minimum size=3,fill=black] (1) at (-2, 2) {}; \node at (-2.7, 2.2) {$g_1g_2r$};
  \node[circle,inner sep=0pt,minimum size=3,fill=black] (1) at (-1.5, 2.5) {}; \node at (0.1, 2.6) {$r=(g_1g_2)^2r$};
 \node[circle,inner sep=0pt,minimum size=3,fill=black] (1) at (-2,-2) {}; \node at (-2.5, -2.2) {$g_2r$};
 \node[circle,inner sep=0pt,minimum size=3,fill=black] (1) at (-1.5,-2.5) {}; \node at (-.5, -2.5) {$g_2g_1g_2r$};  
  \draw [dashed] (-3,1) to (-1.5, 2.5);    \draw [dashed] (-3,-1) to (-1.5, -2.5);
  \draw [dashed, >-] (-1,5.5) to (-1.5,2.5); \draw [dashed, ->] (-1.5,-2.5) to (-1,-5.5);
\node at (-0.9, 5.8) {$\gamma_2$}; \node at (-0.9, -5.8) {$\gamma_2$}; 

 \node[circle,inner sep=0pt,minimum size=3,fill=black] (1) at (-4.5, -2.5) {}; 
  \node[circle,inner sep=0pt,minimum size=3,fill=black] (1) at (-4.5, 2.5) {};
\draw [dashed] (-3,1) to (-4.5, 2.5); \draw [dashed, ->] (-4.5,2.5) to (-5,5.5);  
\draw [dashed] (-3,-1) to (-4.5, -2.5); \draw [dashed, >-] (-5,-5.5) to (-4.5,-2.5); 
\node at (-5.1, 5.8) {$\gamma_1$}; \node at (-5.1, -5.8) {$\gamma_1$}; 

 \node at (0,1.2) {$\Delta'< l(g_2)-l(g_1)$}; \draw [dotted, |-|] (-2.8,0.8) to (-1.3,2.3);
   
   \node at (-4.5, 0) {$l(g_1)=\Delta$}; \draw [dotted, |-|] (-3.3, 1) to (-3.3, -1);

 \node[circle,inner sep=0pt,minimum size=3,fill=black] (1) at (6,-1) {};
 \node[circle,inner sep=0pt,minimum size=3,fill=black] (1) at (6,1) {};
\draw [dashed] (6,-1) to (6,1); \node at (4.2, 1.1) {$(g_1g_2)^2q=q$};
\node at (7.5, -0.9) {$p=g_2g_1g_2q$};
 \node[circle,inner sep=0pt,minimum size=3,fill=black] (1) at (7, 2) {}; \node at (6.2, 2.1) {$g_1g_2q$};
  \node[circle,inner sep=0pt,minimum size=3,fill=black] (1) at (7.5, 2.5) {}; \node at (7.9, 2.5) {$r$};
 \node[circle,inner sep=0pt,minimum size=3,fill=black] (1) at (7,-2) {}; \node at (6.4, -2.1) {$g_2q$};
 \node[circle,inner sep=0pt,minimum size=3,fill=black] (1) at (7.5,-2.5) {};  
  \draw [dashed] (6,1) to (7.5, 2.5);    \draw [dashed] (6,-1) to (7.5, -2.5);
  \draw [dashed, >-] (8,5.5) to (7.5,2.5); \draw [dashed, ->] (7.5,-2.5) to (8,-5.5);
\node at (8.1, 5.8) {$\gamma_2$}; \node at (8.1, -5.8) {$\gamma_2$}; 

 \node[circle,inner sep=0pt,minimum size=3,fill=black] (1) at (4.5, -2.5) {}; 
  \node[circle,inner sep=0pt,minimum size=3,fill=black] (1) at (4.5, 2.5) {};
\draw [dashed] (6,1) to (4.5, 2.5); \draw [dashed, ->] (4.5,2.5) to (4,5.5);  
\draw [dashed] (6,-1) to (4.5, -2.5); \draw [dashed, >-] (4,-5.5) to (4.5,-2.5); 
\node at (3.9, 5.8) {$\gamma_1$}; \node at (3.9, -5.8) {$\gamma_1$}; 

 \node at (9,1.2) {$\Delta' \ge l(g_2)-l(g_1)$}; \draw [dotted, |-|] (6.2,0.8) to (7.7,2.3);
   
   \node at (4.5, 0) {$l(g_1)=\Delta$}; \draw [dotted, |-|] (5.7, 0.9) to (5.7, -1);

\end{tikzpicture}
\caption{$\gamma_1$ and $\gamma_2$ intersect with opposite orientations along a path of length $\Delta=\min\{l(g_1), l(g_2)\}=l(g_1)$. The axes $\gamma_2$ and $g_1\cdot \gamma_2$ intersect along a path $[q,r]$ of length $\Delta' \ge \frac{l(g_2)-l(g_1)}{2}.$} \label{bigint2}
\end{figure}

\end{document}